\documentclass[letterpaper,11pt]{article}
\usepackage[margin=1in]{geometry}  % set the margins to 1in on all sides

\usepackage{listings}
\usepackage{amsmath}
\usepackage{amsthm}
\usepackage{tikz}
\usepackage{caption,subcaption}
\usepackage{array}
\usepackage{mdwmath}
\usepackage{multirow}
\usepackage{mdwtab}
\usepackage{eqparbox}
\usepackage{multicol}
\usepackage{amsfonts}
\usepackage{tikz}
\usepackage{multirow,bigstrut,threeparttable}
\usepackage{amsthm}
\usepackage{array}
\usepackage{bbm}
\usepackage{epstopdf}
\usepackage{mdwmath}
\usepackage{mdwtab}
\usepackage{eqparbox}
\usepackage{tikz}
\usepackage{latexsym}
\usepackage{cite}
\usepackage{amssymb}
\usepackage{bm}
\usepackage{amssymb}
\usepackage{graphicx,algorithm,algorithmic}
\usepackage{mathrsfs}
\usepackage{epsfig}
\usepackage{psfrag}
\usepackage{setspace}
\usepackage[%dvips,
            CJKbookmarks=true,
            bookmarksnumbered=true,
            bookmarksopen=true,
            colorlinks=true,
            citecolor=red,
            linkcolor=blue,
            anchorcolor=red,
            urlcolor=blue
            ]{hyperref}
\usepackage{algorithm}
\usepackage{stfloats}
\input xy
\xyoption{all}

\theoremstyle{plain}
\newtheorem{theorem}{Theorem}
\newtheorem{lemma}{Lemma}
\newtheorem{assumption}{Assumption}

\theoremstyle{definition}

\def\NN{\mathbb{N}}

\def\RR{\mathbb{R}}

\def\calB{\mathcal{B}}

\def\calH{\mathcal{H}}

\def\calN{\mathcal{N}}

\def\calV{\mathcal{V}}

\def\bE{\mathbf{E}}

\def\bP{\mathbf{P}}

\def\bR{\mathbf{R}}

\def\1{\mathbbm{1}}

\usepackage{booktabs}
\usepackage{adjustbox}
\usepackage{multirow}
\usepackage{listings}

\theoremstyle{plain}

\newtheorem{definition}{Definition}

\def \bP {\mathbb{P}}
\def \bE {\mathbb{E}}
\def \bR {\mathbb{R}}

\usepackage{xspace}

\newcommand{\stepa}[1]{\overset{\rm (a)}{#1}}
\newcommand{\stepb}[1]{\overset{\rm (b)}{#1}}
\newcommand{\stepc}[1]{\overset{\rm (c)}{#1}}

\definecolor{myblue}{rgb}{.8, .8, 1}
\definecolor{mathblue}{rgb}{0.2472, 0.24, 0.6} % mathematica's Color[1, 1--3]
\definecolor{mathred}{rgb}{0.6, 0.24, 0.442893}
\definecolor{mathyellow}{rgb}{0.6, 0.547014, 0.24}

\begin{document}
%\begin{frontmatter}

% "Title of the paper"
\title{Minimax Optimal Nonparametric Estimation of Heterogeneous Treatment Effects}
\author{Zijun Gao\thanks{Department of Statistics, Stanford University. Email: \url{zijungao@stanford.edu}.} ~and Yanjun Han\thanks{Department of Electrical Engineering, Stanford University. Email: \url{yjhan@stanford.edu}.}}

\maketitle
%\thankstext{T1}{Footnote to the title with the ``thankstext'' command.}

\begin{abstract}
A central goal of causal inference is to detect and estimate the treatment effects of a given treatment or intervention on an outcome variable of interest, where a member known as the heterogeneous treatment effect (HTE) is of growing popularity in recent practical applications such as the personalized medicine. In this paper, we model the HTE as a smooth nonparametric difference between two less smooth baseline functions, and determine the tight statistical limits of the nonparametric HTE estimation as a function of the covariate geometry. In particular, a two-stage nearest-neighbor-based estimator throwing away observations with poor matching quality is near minimax optimal. We also establish the tight dependence on the density ratio without the usual assumption that the covariate densities are bounded away from zero, where a key step is to employ a novel maximal inequality which could be of independent interest. 
\end{abstract}

\tableofcontents
%\end{frontmatter}

\section{Introduction and Main Results}
Causal inference aims to draw a causal relationship between some treatment and target responses. Nowadays, personalized medicine and huge available data make heterogeneous treatment effect (HTE) estimation meaningful and possible. While there are various practical approaches of estimating the HTE \cite{athey2016recursive,powers2017some,wager2018estimation,Kunzel4156}, some important theoretical questions remain unanswered.

In this paper, we consider the Neyman-Rubin potential outcome model \cite{rubin1974estimating} for the treatment effect. Assume for simplicity that there are $n$ individuals in the treatment group and the control group, respectively, where the generalizations to different group sizes are straightforward. For each individual $i\in [n]$ in the control group, we observe a vector of covariates $X_i^0\in \bR^d$ and her potential outcome $Y_i^0\in \bR$ for not receiving the treatment. Similarly, for individual $i$ in the treatment group, the covariates $X_i^1$ and the potential outcome $Y_i^1$ under the treatment are observed. We assume the following model for the potential outcomes: for any $i\in [n]$, 
\begin{align}
Y_i^0 = \mu_0(X_i^0) + \varepsilon_i^0, \qquad Y_i^1 = \mu_1(X_i^1) + \varepsilon_i^1, \label{eq.treatment_model}
\end{align}
where $\mu_0, \mu_1: \bR^d \to \bR$ are the baseline functions for the control and treatment groups, respectively, and $\varepsilon_i^0, \varepsilon_i^1$ are modeling errors. The heterogeneous treatment effect $\tau$ is defined to be the difference of the baseline functions: 
\begin{align}\label{eq.HTE}
\tau(x) \triangleq \mu_1(x) - \mu_0(x). 
\end{align}
In other words, the treatment effect $\tau(x)$ is the expected change in the outcomes after an individual with covariate $x$ receives the treatment, which is usually heterogeneous as $\tau(x)$ typically varies in $x$. The target is to find an estimator $\hat{\tau}$ which comes close to the true HTE $\tau$ under the $L_2$ norm of functions, based on the control and treatment observations $\{(X_i^0, Y_i^0)\}_{i\in [n]}, \{(X_i^1, Y_i^1)\}_{i\in [n]}$. 

In HTE estimation, modeling of the baseline functions $\mu_0, \mu_1$ or the HTE $\tau$ plays an important role. In practice, the treatment effect $\tau$ is typically easier to estimate than the baseline functions $\mu_0, \mu_1$, as ideally the treatment effect depends solely on the single treatment. In this paper, we assume that both the baseline and treatment effect functions are nonparametric functions, with an additional constraint that $\tau$ is smoother than $\mu_0$: 
\begin{assumption}[Baseline and HTE functions]
	\label{assu:smoothness}
	The baseline $\mu_0$ and the treatment effect $\tau$ belong to $d$-dimensional H\"{o}lder balls with smoothness parameters $\beta_\mu\le \beta_\tau$, respectively. 
\end{assumption}

Recall the following definition of H\"{o}lder balls. 
\begin{definition}[H\"{o}lder ball]
	The H\"{o}lder ball $\calH_d(\beta)$ with dimension $d$ and smoothness parameter $\beta\ge 0$ is the collection of functions $f: \bR^d\to \bR$ supported on $[0,1]^d$ with
	\begin{align*}
	\left| \frac{\partial^s f}{\partial x_1^{\beta_1}\cdots \partial x_d^{\beta_d} }(x) - \frac{\partial^s f}{\partial x_1^{\beta_1}\cdots \partial x_d^{\beta_d} }(y) \right|\le L\|x-y\|_2^\alpha
	\end{align*}
	for all $x,y\in \bR^d$ and all multi-indices $(\beta_1,\cdots,\beta_d)\in \NN^d$ with $\sum_{i=1}^d \beta_i = s$, where $\beta = s+\alpha$ with $s\in \NN, \alpha\in (0,1]$. Throughout we assume that the radius $L$ is a fixed positive constant and omit the dependence on $L$. 
\end{definition}

Assumption \ref{assu:smoothness} imposes no structural assumptions on $\mu_0$ and $\tau$ except for the smoothness, and assumes that the HTE $\tau$ is smoother and thus easier to estimate than the baseline $\mu_0$. This typically holds as intuitively fewer factors contribute to the HTE than baselines, and semiparametric models usually assume simple forms of HTEs such as constant or linear \cite{kennedy2016semiparametric,chernozhukov2018double}, where Assumption \ref{assu:smoothness} could be treated as a nonparametric counterpart. Standard results in nonparametric estimation reveal that if one na\"{i}vely estimates the baselines $\mu_0$ and $\mu_1$ separately based on \eqref{eq.treatment_model}, then each function can be estimated within accuracy $\Theta(n^{-\beta_\mu / (2\beta_\mu + d)})$, and so is the difference $\tau$ in \eqref{eq.HTE}. However, if the covariates in the control and treatment groups match perfectly, i.e. $X_i^0 = X_i^1$ for all $i\in [n]$, then differencing \eqref{eq.treatment_model} gives an improved estimation accuracy $\Theta(n^{-\beta_\tau / (2\beta_\tau + d)})$ for the HTE $\tau$. In general, the estimation performance of HTE is an interpolation between the above extremes and depends heavily on the quality of the covariate matching. To model such qualities, we have the following assumption on the covariates $X_i^0, X_i^1$. 
\begin{assumption}[Covariates]\label{assu:covariate}
	The covariates are generated under fixed or random designs below: 
	\begin{itemize}
		\item \emph{Fixed design:} the covariates are generated from the following fixed grid
		\begin{align*}
		\left\{X_i^0 \right\}_{i\in [n]} &= \left\{ 0, \frac{1}{m}, \cdots, \frac{m-1}{m} \right\}^d,\quad \left\{X_i^1 \right\}_{i\in [n]} = \left\{ 0, \frac{1}{m}, \cdots, \frac{m-1}{m} \right\}^d + \Delta,
		\end{align*}
		for some vector $\Delta \in \bR^d$ with $\|\Delta\|_\infty \le 1/(2m)$, and $m\triangleq n^{1/d}$ is assumed to be an integer. 
		\item \emph{Random design:} the covariates are i.i.d. sampled from unknown densities $g_0, g_1$ on $[0,1]^d$: 
		\begin{align*}
		X_1^0, X_2^0, \cdots, X_n^0 \overset{\text{\rm i.i.d.}}{\sim} g_0, \quad X_1^1, X_2^1, \cdots, X_n^1 \overset{\text{\rm i.i.d.}}{\sim} g_1,
		\end{align*}
		where there is a bounded likelihood ratio $\kappa^{-1}\le g_0(x)/g_1(x)\le \kappa$  on the densities.
	\end{itemize}
\end{assumption}

Under the fixed design, the covariates in both groups are evenly spaced grids in $[0,1]^d$, with a shift $\Delta$ quantifying the matching distance between the control and treatment groups. The fixed design is not very practical, but the analysis will provide important insights for the HTE estimation. The random design model is more meaningful and realistic without any matching parameter, and the density ratio $g_0/(g_0+g_1)$ corresponds to the propensity score, a key quantity in causal inference. We assume that $\kappa^{-1} \le g_0/g_1 \le \kappa$ everywhere, as it is usually necessary to have a propensity score bounded away from $0$ and $1$. Besides the bounded likelihood ratio, we remark that it is \emph{not} assumed that the densities $g_0, g_1$ are bounded away from zero. 

Finally it remains to model the noises $\varepsilon_i^0, \varepsilon_i^1$, and we assume the following mild conditions. 
\begin{assumption}[Noise]\label{assu:noise}
	The noises $\varepsilon_i^0, \varepsilon_i^1$ are mutually independent, zero-mean and of variance $\sigma^2$. 
\end{assumption}

\begin{figure*}[t]
	\centering
	\begin{subfigure}[b]{0.45\textwidth}
		\includegraphics[width=\linewidth]{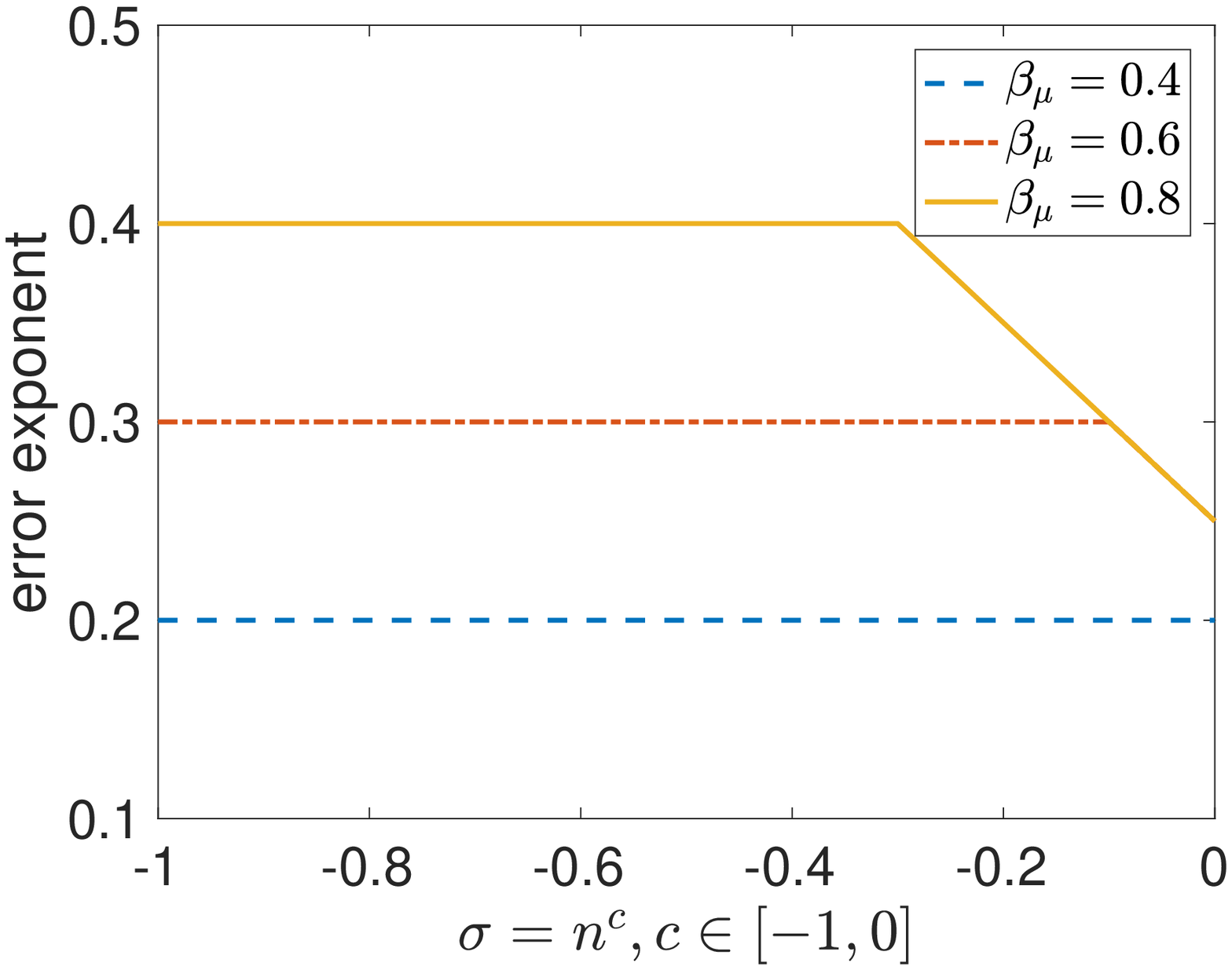}\caption{Fixed design.}
	\end{subfigure}
	\begin{subfigure}[b]{0.45\textwidth}
		\includegraphics[width=\linewidth]{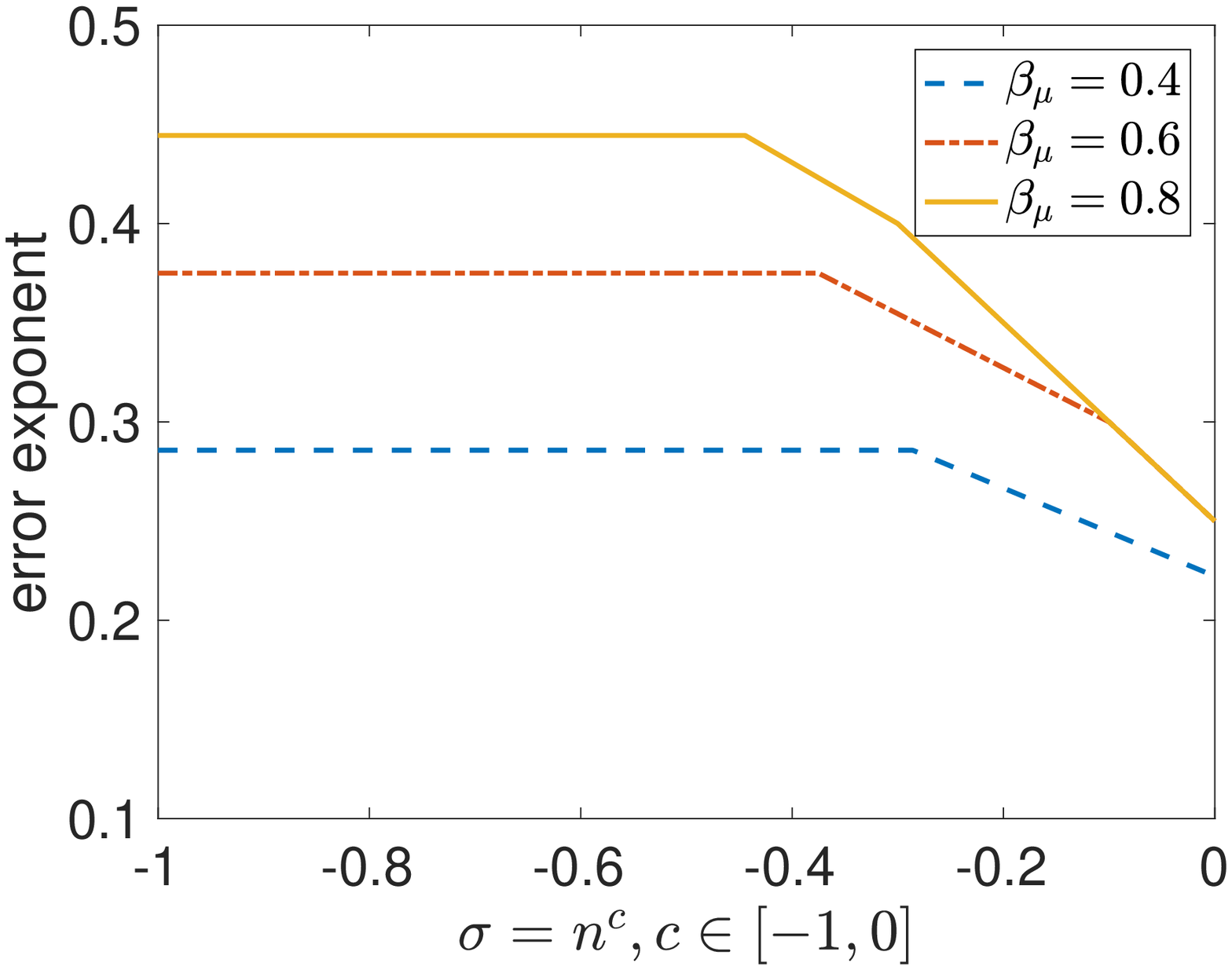}\caption{Random design.}
	\end{subfigure}
	\caption{Negated error exponent of the minimax rate of HTE estimation under both fixed and random designs, with $d=2$, $\beta_\mu \in \{0.4, 0.6, 0.8\}$, $\beta_\tau = 1$, $\kappa=\Theta(1)$, and $\sigma = n^c$ with $c\in [0,1]$. The error exponents exhibit at most two regimes under the fixed design, and three regimes under the random design.}
\end{figure*}\label{fig.minimax}

Based on the above assumptions, the target of this paper is to characterize the minimax risks of the nonparametric HTE estimation under both fixed and random designs. Specifically, we are interested in the following minimax $L_1$ risk
\begin{align*}
R_{n,d,\beta_\mu,\beta_\tau,\sigma}^{\text{fixed}}(\Delta) \triangleq \inf_{\hat{\tau}}\sup_{\mu_0\in \calH_d(\beta_\mu), \tau\in \calH_d(\beta_\tau)} \bE_{\mu_0,\tau} [ \|\hat{\tau} - \tau\|_1 ] 
\end{align*}
for fixed design with matching parameter $\Delta\in \bR^d$, where the infimum is taken over all possible estimators $\hat{\tau}$ based on the observations $\{(X_i^0, Y_i^0)\}_{i\in [n]}, \{(X_i^1, Y_i^1)\}_{i\in [n]}$. As for the minimax risk for random designs, we include the dependence on the density ratio $\kappa$ and use the $L_1$ norm based on the density $g_0$ (as we do not assume that the densities are bounded away from zero): 
\begin{align*}
R_{n,d,\beta_\mu,\beta_\tau,\sigma}^{\text{\rm random}}(\kappa)  \triangleq \inf_{\hat{\tau}}\sup_{\mu_0\in \calH_d(\beta_\mu), \tau\in \calH_d(\beta_\tau), g_0, g_1} \bE_{\mu_0,\tau} [ \|\hat{\tau} - \tau\|_{L_1(g_0)} ]. 
\end{align*}
The main aim of this paper is to characterize the tight minimax rates for the above quantities, and in particular, how they are determined by the covariate geometry and the full set of parameters $(n,\kappa,\sigma)$ of interest. In addition, we aim to extract useful practical insights (instead of fully practical algorithms) based on the minimax optimal estimation procedures established in theory. Moreover, we mostly focus on the special nature of the HTE estimation problem instead of general nonparametric estimation, and therefore we elaborate less on broad issues of nonparametric statistics such as adaptation/hyperparameter estimation and refer interested readers to known literature (see Section \ref{sec:discussion}). 

Our first result is the characterization of the minimax rate for HTE estimation under fixed designs. 
\begin{theorem}[Fixed Design]\label{thm.minimax_fixed}
	Under Assumptions \ref{assu:smoothness}--\ref{assu:noise},
	\begin{align*}
	R_{n,d,\beta_\mu,\beta_\tau,\sigma}^{\text{\rm fixed}}(\Delta) \asymp n^{-\frac{\beta_\mu}{d}}(n^{\frac{1}{d}}\|\Delta\|_\infty)^{\beta_\mu \wedge 1} + \left(\frac{\sigma^2}{n}\right)^{\frac{\beta_\tau}{2\beta_\tau + d}}. 
	\end{align*}
\end{theorem}

Theorem \ref{thm.minimax_fixed} shows that, as the covariate matching quality improves (i.e. $\|\Delta\|_\infty$ shrinks), the estimation error of the HTE decreases from $n^{-\beta_\mu/d}$, which is slightly better than the estimation error $n^{-\beta_\mu/(2\beta_\mu+d)}$ for the baselines, to the optimal estimation error $(\sigma^2/n)^{\beta_\tau/(2\beta_\tau+d)}$ when the learner has $n$ direct samples from $\tau$. We also remark that the covariate matching quality is determined by the $\ell_\infty$ norm of $\Delta$, and the matching bias does not depend on the noise level $\sigma$. 

The minimax rate for HTE estimation under random designs exhibits more interesting behaviors. 
\begin{theorem}[Random Design]\label{thm.minimax_random}
	Under Assumptions \ref{assu:smoothness}--\ref{assu:noise}, if $\beta_\tau\le 1$ and $\kappa\le n$, then
	\begin{align*}
	R_{n,d,\beta_\mu,\beta_\tau,\sigma}^{\text{\rm random}}(\kappa) = \widetilde{\Theta}\left( \left(\frac{\kappa}{n^2}\right)^{\frac{1}{d(\beta_\mu^{-1} + \beta_\tau^{-1})}} + \left(\frac{\kappa\sigma^2}{n^2}\right)^{\frac{1}{2+d(\beta_\mu^{-1} + \beta_\tau^{-1})}} + \left(\frac{\kappa\sigma^2}{n}\right)^{\frac{\beta_\tau}{2\beta_\tau + d}} \right). 
	\end{align*}
\end{theorem}
Theorem \ref{thm.minimax_random} provides the first minimax analysis of the HTE estimation with tight dependence on all parameters $(n,\kappa,\sigma)$ without assuming densities bounded away from zero. Interestingly, Theorem \ref{thm.minimax_random} shows that there are three regimes of the minimax rate of the HTE estimation under random designs, and in particular, there is an intermediate regime where neither the matching bias nor the estimation of $\tau$ dominates. Moreover, when it comes to the dependence on the density ratio $\kappa$, the effective sample size is $n/\sqrt{\kappa}$ in the first two regimes, while it becomes $n/\kappa$ in the last regime. Examples of the minimax rates of HTE estimation under fixed and random designs are illustrated in Figure \ref{fig.minimax}. 

\subsection{Related works}
The nonparametric regression of a single function has a long history and is a well-studied topic in nonparametric statistics, with various estimators including kernel smoothing \cite{rosenblatt1956remarks}, local polynomial fitting \cite{stone1977consistent,fan1997local}, splines \cite{de1978practical,green1993nonparametric}, wavelets \cite{donoho1995wavelet}, nearest neighbors \cite{fix1951discriminatory,stone1977consistent}, and so on. We refer to the excellent books \cite{nemirovski2000topics,gyorfi2006distribution,tsybakov2008introduction,johnstone2011gaussian,biau2015lectures} for an overview. However, estimating the difference of two nonparametric functions remains largely underexplored. 

In the history of causal inference, the average treatment effect (ATE) has long been the focus of research \cite{imbens2015causal}. Recently, an increasing number of works seek to estimate the HTE, and there are two major threads of approaches: parametric (semi-parametric) estimation and nonparametric estimation. 

In parametric estimation, a variety of parametric assumptions such as linearity are imposed on the baseline functions and HTE. The problem then reduces to classic parametric estimation and many methods are applicable, e.g. \cite{Lu2014Simple,imai2013estimating}. In this paradigm, classic parametric rate is obtained under proper regularity assumptions. To further deal with observational study and bias due to the curse of high-dimensionality, \cite{chernozhukov2018double} develops a double/debiased procedure allowing a less accurate estimation of the mean function without sacrificing the parametric rate. 

In nonparametric estimation, there are a large number of practical approaches proposed, such as spline-based method \cite{chen2007large}, matching-based method \cite{xie2012estimating} and tree-based method \cite{athey2016recursive,powers2017some,wager2018estimation}. However, relatively fewer work study the statistical limit of nonparametric estimation of HTE. One work \cite{nie2017quasi} argued that a crude estimation of baselines suffices for the optimal estimation of HTE, while their approach did not take into account the impact of the covariate geometry. Some works \cite{alaa2018limits,Kunzel4156} studied the minimax risk with smoothness assumptions on the baseline functions, but they did not directly model the smoother HTE function and thus arrived at a relatively easier claim that the difficulty is determined by the less smooth baseline. In contrast, we assume a smoother HTE and provide additional insights on how to construct pseudo-observations and discarding poor-quality data based on the covariate geometry. Hence, to the best of our knowledge, this paper establishes the first minimax rates with tight dependence on both the covariate geometry, and the parameters $\kappa$ related to propensity scores as well as $\sigma$ related to noise levels. 

We also comment that in nonparametric regression problems with random design, a uniform lower bound on the density is usually necessary in most past works \cite{brown2002asymptotic,gyorfi2006distribution}. To overcome this difficulty, there is a recent line of research showing that to estimate nonparametric functionals, Hardy--Littlewood-type maximal inequalities are often useful to deal with densities close to zero for both kernel methods \cite{han2017optimal} and nearest neighbors \cite{jiao2018nearest}. Therefore, this work is a continuation of the above thread with a novel maximal inequality on another different problem.  

\subsection{Notations}
Let $\RR, \NN$ be the set of all real numbers and non-negative integers, respectively. For $p\in [1,\infty]$, we denote by $\|\cdot\|_p$ both the $\ell_p$ norm for vectors and the $L_p$ norm for functions. Let $[n]\triangleq \{1,2,\cdots,n\}$, and $a\wedge b\triangleq \min\{a,b\}$. We call $K: \bR^d\to \bR$ a kernel of order $k$ if $\int K(x)dx =1$ and $\int (v^\top x)^\ell K(x)dx=0$ for any $v\in \bR^d$ and $1\le\ell\le k$. For non-negative sequences $\{a_n\}$ and $\{b_n\}$, we write $a_n\asymp b_n$, or $a_n = \Theta(b_n)$, to denote that $cb_n \le a_n \le Cb_n$ for some absolute constants $c,C>0$ independent of $n$. The notation $a_n = \widetilde{\Theta}(b_n)$ means that the above inequality holds within multiplicative polylogarithmic factors in $n$. 

\subsection{Organization}
In Section \ref{sec:fixedDesign}, we construct the minimax optimal HTE estimator via a combination of kernel methods and covariate matching under the simple fixed design setting. For the random design, a two-stage nearest-neighbor based algorithm is proposed in Section \ref{sec:randomDesign} to trade off two types of biases and the variance. The efficacy of the proposed estimator is shown via numerical experiments in Section \ref{sec:experiment}. Some limitations and future works are discussed in Section \ref{sec:discussion}, and all proofs are relegated to the Appendices. 

\section{Fixed Design}\label{sec:fixedDesign}
This section is devoted to the minimax rate of HTE estimation under fixed design. To construct the estimator, we first form pseudo-observations of the outcomes in the treatment group on the covariates in the control group based on covariate matching, and then apply the classic kernel estimator based on perfectly matched pseudo-observations. We also sketch the ideas of the minimax lower bound and show that both the matching bias and the estimation error are inevitable. 

\subsection{Estimator Construction}
Recall that in the perfect matching case, i.e. $\Delta=0$, the HTE function $\tau$ can be estimated by classic nonparametric estimators after taking the outcome difference $Y_i^1 - Y_i^0$ in \eqref{eq.treatment_model}. This basic idea will be used for general fixed designs: first, for each $i\in [n]$, we form pseudo-observations $\hat{Y}_i^1$ with target mean $\mu_1(X_i^0)$, where the observed covariate $X_i^1$ in the treatment group moves to the nearest covariate $X_i^0$ in the control group. Next, we apply the kernel estimator in the perfect matching case to the pseudo-difference $\hat{Y}_i^1 - Y_i^0$. The covariate matching step also makes use of a suitable kernel method, where the matching performance depends on the smoothness of the baseline function $\mu_0$. 

Specifically, the pseudo-observations $\hat{Y}_i^1$ are constructed as follows. Let $t \triangleq \lfloor \beta_\mu \rfloor + 1$, and fix a covariate $x$ in the set of control covariates. For each $j\in [d]$, let $x_{1,j},\cdots,x_{t,j}$ be the $t$ closest grid points in $\{\Delta_j, 1/m+\Delta_j, \cdots 1-1/m+\Delta_j\}$ to $x_{j}$, the $j$-th coordinate of $x$. Moreover, for each coordinate $j\in [d]$, we also compute the following weights $w_{1,j},\cdots,w_{t,j}$ such that 
\begin{align}\label{eq:weight}
\sum_{i=1}^t w_{i,j} (x_{i,j} - x_{j})^\ell = \mathbbm{1}(\ell = 0), \quad \forall 0\le \ell \le \lfloor \beta_\mu \rfloor.
\end{align}

To form the pseudo-observation of the treatment outcome $\hat{Y}^1(x)$ on the covariate $x$, we use the treatment observations $Y^1(x_i)$ for all multi-indices $i=(i_1,\cdots,i_d)\in [t]^d$ and $x_i \triangleq (x_{i,1},x_{i,2},\cdots,x_{i,d})$, with $w_i = w_{i,j}$ being the weight of the covariate $x_i$. Note that each $x_i$ belongs to the covariate grid of the treatment group and therefore $Y^1(x_i)$ is observable. Finally, given the above covariates $x_i$ and weights $w_i$, the pseudo-observation of the treatment outcome at the control covariate $x$ is 
\begin{align}\label{eq:pseudo}
\hat{Y}^1(x) = \sum_{i \in [t]^d} w_i Y^1(x_i)
\end{align}
In approximation theory, this is an interpolation estimate of $\mu_1(x)$ based on the function values $\{\mu_1(x_i)\}_{i\in [t]^d}$ in the neighborhood of $x$. 

Given the pseudo-observations, the final HTE estimator $\hat{\tau}(x_0)$ for any $x_0\in [0,1]^d$ is defined as
\begin{align}\label{eq.fixed_design_estimator}
\hat{\tau}(x_0) \triangleq \frac{\sum_{i=1}^n K((X_i^0-x_0)/h) (\hat{Y}^1(X_i^0) - Y_i^0) }{\sum_{i=1}^n  K((X_i^0-x_0)/h)},
\end{align}
i.e. the Nadaraya-Watson estimator \cite{nadaraya1964estimating,watson1964smooth} applied to the pseudo-differences, where $K$ is any kernel of order $\lfloor \beta_\tau\rfloor$, and $h\in (0,1)$ is a suitable bandwidth. The complete description of the estimator construction is displayed in Algorithm 1. 

\begin{algorithm}[tb]
	\caption{Estimator Construction under Fixed Design}\label{alg:fixed}
	\begin{algorithmic}
		\STATE {\bfseries Input:} control observations $\{(X_i^0, Y_i^0)\}_{i\in [n]}$, treatment observations $\{(X_i^1, Y_i^1)\}_{i\in [n]}$, dimensionality $d$,  kernel $K$ of order $\lfloor \beta_\tau \rfloor$, bandwidth $h>0$, query $x_0\in [0,1]^d$. 
		\STATE {\bfseries Output:} HTE estimator $\hat{\tau}(x_0)$ at point $x_0$. 
		\FOR{$x\in \{X_i^0\}_{i\in [n]}$}
		\STATE Choose $t = \lfloor \beta_\mu \rfloor + 1$, and find neighboring treatment covariates $\{x_i\}_{i\in [t]^d}$; 
		\STATE Compute the weights $\{w_i\}_{i\in [t]^d}$ according to \eqref{eq:weight}; 
		\STATE Compute the pseudo-observation $\hat{Y}^1(x)$ according to \eqref{eq:pseudo}. 
		\ENDFOR
		\STATE Output the final estimator $\hat{\tau}(x_0)$ according to \eqref{eq.fixed_design_estimator}.
	\end{algorithmic}
\end{algorithm}

%\begin{remark}
%	Note that in principle one may use different kernels for different $x_0$ in \eqref{eq.fixed_design_estimator}, therefore the boundary effects can be handled in the same way as \cite{lepski1999estimation} by employing different kernels near the boundary. 
%\end{remark}

The performance of the above estimator is summarized in the following theorem. 
\begin{theorem}\label{thm.fixed_upper}
	There exists a constant $C>0$ independent of $(n,h)$ such that
	\begin{align*}
	\sup_{\mu_0\in \calH_d(\beta_\mu), \tau\in \calH_d(\beta_\tau)} \bE_{\mu_0,\tau} [ \|\hat{\tau} - \tau\|_1 ] \le C\left(n^{-\frac{\beta_\mu}{d}}(n^{\frac{1}{d}}\|\Delta\|_\infty)^{\beta_\mu\wedge 1} + h^{\beta_\tau} + \frac{\sigma}{\sqrt{nh^d}}\right).
	\end{align*}
	In particular, the estimator $\hat{\tau}$ in \eqref{eq.fixed_design_estimator} achieves the upper bound in Theorem \ref{thm.minimax_fixed} for $ h = \Theta(n^{-1/(2\beta_\tau + d)})$. 
\end{theorem}

We sketch the proof idea of Theorem \ref{thm.fixed_upper} here. Using the definition of pseudo-observations $\hat{Y}^1$, for each $i\in [n]$ we have $
\hat{Y}^1(X_i^0) - Y_i^0 = \tau(X_i^0) + b_\mu(X_i^0) + \tilde{\varepsilon}_i,
$
where $\tau(X_i^0)$ is the target treatment effect at the point $X_i^0$, $b_\mu(x) = \sum_{i\in [t]^d} w_i \mu_1(x_i) - \mu_1(x)$ is the matching bias incurred by the linear interpolation, and $\tilde{\varepsilon}_i$ is a linear combination of the error terms. Then the first term of Theorem \ref{thm.fixed_upper} is an upper bound of the matching bias, while the other terms come from the traditional bias-variance tradeoff in nonparametric estimation. We leave the full proof to the appendix. 

%The first error term of Theorem \ref{thm.fixed_upper} is an upper bound of the matching bias, or the interpolation bias. In fact, by Taylor expansion and the H\"{o}lder ball assumption, the interpolation bias of the baseline function satisfies
%\begin{align*}
%|b_\mu(x)|&\triangleq \left|\sum_{i\in [t]^d} w_i \mu_1(x_i) - \mu_1(x)\right| \le C\sum_{\beta}  \prod_{j=1}^d \left(\sum_{i=1}^t |w_{i,j}|\cdot |x_j - x_{i,j}|^{\beta_j}\right),
%\end{align*}
%where the outer summation is over vectors $\beta=(\beta_1,\cdots,\beta_d)$ summing into $\beta_\mu$ with non-negative entries and at most one non-integer entry. By solving the linear system \eqref{eq.weight_single_dim}, one may show that (cf. Lemma \ref{lemma.weights} in the Appendix) $|w_{i,j}|\le C$ for all $i,j$ and $|w_{i,j}|\le Cm\Delta_j$ as long as $|x_j - x_{i,j}|\neq \Delta_j$ (i.e. $x_{i,j}$ is not the nearest neighbor of $x_j$). Consequently, it holds that $|b_\mu(x)|\le C\|\Delta\|_\infty^{\beta_\mu}$ when $\beta_\mu \le 1$ and $|b_{\mu}(x)|\le Cn^{(1-\beta_\mu)/d}\|\Delta\|_\infty$ when $\beta_\mu>1$ after some algebra. 
%
%Next, for each $i\in [n]$ we have
%\begin{align*}
%\hat{Y}^1(X_i^0) - Y_i^0 = \tau(X_i^0) + b_\mu(X_i^0) + \tilde{\varepsilon}_i,
%\end{align*}
%where $\tilde{\varepsilon}_i$ is a linear combination of the error terms. Adapting the classic nonparametric analysis to the Nadaraya-Watson estimator in the above model, the bias-variance tradeoff on the bandwidth $h$ will be $h^{\beta_\mu} + \sigma/\sqrt{nh^d}$, as desired. 

\subsection{Minimax Lower Bound}
In this section, we show that the above HTE estimator is minimax optimal via the following minimax lower bound. 
\begin{theorem}\label{thm.fixed_lower}
	There exists a constant $c>0$ independent of $n$ such that for any HTE estimator $\hat{\tau}$, it holds that under Gaussian noise, 
	\begin{align*}
	\sup_{\mu_0\in \calH_d(\beta_\mu), \tau\in \calH_d(\beta_\tau)} \bE_{\mu_0,\tau} [ \|\hat{\tau} - \tau\|_1 ] \ge c\left(n^{-\frac{\beta_\mu}{d}}(n^{\frac{1}{d}}\|\Delta\|_\infty)^{\beta_\mu\wedge 1} + \left(\frac{\sigma^2}{n}\right)^{\frac{\beta_\tau}{2\beta_\tau+d}} \right).
	\end{align*}
\end{theorem}
As for the proof of Theorem \ref{thm.fixed_lower}, note that the estimation error $(\sigma^2/n)^{\beta_\tau/(2\beta_\tau+d)}$ is optimal in the nonparametric estimation of $\tau$ even if there is a perfect covariate matching. Hence, it remains to prove the first term, i.e. the matching bias. The proof is based on Le Cam's two-point method \cite{yu1997assouad} and a functional optimization problem. Consider two scenarios $(\mu_0, \tau)$ and $(\mu_0', \tau')$ with $\sigma=0$ and $\mu_0'(x) \equiv 0, \tau'(x) \equiv 0$, where $(\mu_0(x), \tau(x))$ is the solution to the following optimization problem: 
\begin{align}\label{eq.optimization}
\begin{split}
\text{maximize} \qquad &\|\tau\|_1 \\
\text{subject to}\qquad &\mu_0(X_i^0) = 0, \quad \mu_0(X_i^1) + \tau(X_i^1) = 0, \quad i=1,\cdots,n, \\
& \mu_0 \in \calH_d(\beta_\mu), \tau\in \calH_d(\beta_\tau). 
\end{split}
\end{align}
Note that when $(\mu_0, \tau)$ is a feasible solution to \eqref{eq.optimization}, under both scenarios $(\mu_0, \tau)$ and $(\mu_0', \tau')$ all outcomes are identically zero in both groups. Hence, these scenarios are completely indistinguishable, and it remains to show that the objective value of \eqref{eq.optimization} is at least the first error term. We defer the complete proofs to the appendix. 

\section{Random Design}\label{sec:randomDesign}
The random design assumption is more practical and complicated than the fixed design, as the nearest matching distances of different individuals are typically different without the regular grid structure. In particular, under the random design some covariates have better matching quality than others, in the sense that the minimum pairwise distance between the control and treatment covariates is $\Theta(n^{-2/d})$ while the typical distance is as large as $\Theta(n^{-1/d})$. In this section, we propose a two-stage nearest-neighbor estimator for HTE under the random design, and show that the minimax estimation error exhibits three different behaviors. 

\subsection{Estimator Construction}
In the previous section, we construct pseudo-observations at each control covariate using the same treatment grid geometry in the neighborhood of that covariate, thanks to the fixed grid design assumption. However, the local geometries in the random design may differ spatially. For example, for $n$ i.i.d. uniformly distributed random variables on $[0,1]$, the typical spacing between spatially adjacent values is $\Theta(n^{-1})$, while the minimum spacing is of the order $\Theta(n^{-2})$ with high probability. Therefore, the key insight behind the new estimator construction is to look for the best geometry in each local neighborhood. Specifically, we propose the following two-stage estimator: at the first stage we collect $m_1$ nearest control covariates of the query point, while at the second stage we find the $1$-nearest-neighbor treatment covariates for all above control samples and only pick $m_2\le m_1$ of them with desirably small nearest neighbor distances. Note that the nearest-neighbor estimator is a better candidate than the kernel-based estimator here, as the nearest-neighbor distance can adapt to different density regimes and result in adaptive bandwidths. 

\begin{algorithm}[tb]
	\caption{Estimator Construction under Random Design}\label{alg:random}
	\begin{algorithmic}
		\STATE {\bfseries Input:} control observations $\{(X_i^0, Y_i^0)\}_{i\in [n]}$, treatment observations $\{(X_i^1, Y_i^1)\}_{i\in [n]}$, dimensionality $d$, number of neighbors $m_1, m_2$ with $m_1\ge \kappa m_2$, query point $x_0\in [0,1]^d$. 
		\STATE {\bfseries Output:} HTE estimator $\hat{\tau}(x_0)$ at query point $x_0$. 
		\STATE Find $m_1$ nearest control covariates $ \{X_i^0\}_{i\in I_1}$ in the Euclidean distance to $x_0$, with $|I_1| = m_1$; 
		\FOR{$i\in I_1$}
		\STATE Compute the minimum distance $d_i = \min_{j\in [n]} \|X_i^0 - X_j^1 \|_2$ to the treatment covariate;
		\STATE Let $j(i)\in [n]$ be any minimizer to the above minimization program. 
		\ENDFOR
		\STATE Find the smallest $m_2$ elements of $(d_i)_{i\in I_1}$ and denote the index set by $I_2$, with $|I_2| = m_2$; 
		\STATE Return the final estimate $\hat{\tau}(x_0) =  m_2^{-1}\sum_{i\in I_2} (Y_{j(i)}^1 - Y_i^0)$. 
	\end{algorithmic}
\end{algorithm}

%\begin{figure}[h]
%	\centering
%	\includegraphics[scale=0.4]{bandwidth.eps}
%	\caption{The optimal exponents of the bandwidths $h_1$ (dashed lines) and $h_2$ (solid lines) under the random design, with $d=2$, $\beta_\tau\in\{0.4,0.6,0.8\}$, $\beta_\tau=1$, and $\sigma = n^{c}$ with $c\in [-1,0]$.}\label{fig.bandwidth}
%\end{figure}

The detailed estimator construction is displayed in Algorithm \ref{alg:random}. Note that the first stage is similar in spirit to the kernel estimator, where we simply collect $m_1$ control samples in the small neighborhood and aim to take the (weighted) average. However, the second stage is very different in the sense that it asks to throw away samples with poor covariate matching quality to all treatment covariates and only keeps $m_2$ pairs of samples. Finally, a simple average is taken to the difference of the selected $m_1$ pairs, which is sufficient for smoothness $\beta_\tau \le 1$. Although Algorithm \ref{alg:random} can be applied to higher smoothness in principle, we will discuss the fundamental challenges of $\beta_\tau>1$ in Section \ref{sec:discussion}. 

The optimal choices of $(m_1,m_2)$ depend on the density ratio $\kappa$ and noise level $\sigma$ as follows: 
\begin{align*}
(m_1, m_2) = \widetilde{\Theta}\begin{cases}
(\kappa^{\frac{\beta_\mu}{\beta_\tau+\beta_\mu}}n^{\frac{\beta_\tau - \beta_\mu}{\beta_\tau + \beta_\mu}},1) & \text{if } \sigma \le \sigma_1\\
(n\cdot (\kappa\sigma^2/n^2)^{\frac{d\beta_\mu}{2\beta_\mu\beta_\tau + d(\beta_\mu+\beta_\tau)}},  (n^2/\kappa)^{\frac{2\beta_\mu\beta_\tau}{2\beta_\mu\beta_\tau + d(\beta_\mu+\beta_\tau)}}\sigma^{\frac{2d(\beta_\mu+\beta_\tau)}{2\beta_\mu\beta_\tau + d(\beta_\mu+\beta_\tau)}}  ) & \text{if } \sigma_1 < \sigma \le \sigma_2 \\
(n^{\frac{2\beta_\tau}{2\beta_\tau+d}} (\sigma^2\kappa)^{\frac{d}{2\beta_\tau+d}},(n/\kappa)^{\frac{2\beta_\tau}{2\beta_\tau+d}} \sigma^{\frac{2d}{2\beta_\tau+d}}) & \text{if }\sigma > \sigma_2
\end{cases}
\end{align*}
with noise thresholds $(\sigma_1, \sigma_2)$ given by 
\begin{align*}
\sigma_1 = (\kappa/n^2)^{1/d(\beta_\mu^{-1} + \beta_\tau^{-1})}, \quad \sigma_2 = n^{(\beta_\tau-\beta_\mu)/(2\beta_\tau)-\beta_\mu/d}/\sqrt{\kappa}. 
\end{align*}

The rationale behind the above choices lies in the following theorem. 

\begin{theorem}\label{thm.random_upper}
	Given Assumptions \ref{assu:smoothness}-\ref{assu:noise} and $\beta_\mu\le\beta_\tau\le 1$, fix any integer parameters $(m_1, m_2)$ with $m_1\ge \kappa m_2$ and $m_2 = \Omega(\log n)$, the following upper bound holds for the estimator $\hat{\tau}$ in Algorithm \ref{alg:random}: 
	\begin{align*}
	&\sup_{\mu_0\in \calH_d(\beta_\mu), \tau\in \calH_d(\beta_\tau)} \bE_{\mu_0,\tau} [ \|\hat{\tau} - \tau\|_{L_1(g_0)} ] \le \mathsf{polylog}(n)\cdot \left( \left(\frac{\kappa m_2}{nm_1}\right)^{\frac{\beta_\mu}{d}} + \left(\frac{m_1}{n}\right)^{\frac{\beta_\tau}{d}} + \frac{\sigma}{\sqrt{m_2}}  \right),
	\end{align*}
	where $\mathsf{polylog}(n)$ hides poly-logarithmic factors in $n$. In particular, if the parameters $(m_1,m_2)$ are chosen as above, the estimator $\hat{\tau}$ achieves the upper bound of Theorem \ref{thm.minimax_random}. 
\end{theorem}

We explain the originations of the three types of errors in Theorem \ref{thm.random_upper}. The first error comes from the covariate matching bias, where it is given by the average of the $m_2$ smallest $1$-nearest-neighbor distances between control-treatment pairs at the second stage. The second error comes from the bias incurred at the first stage, where the key quantity of interest is the distance between a random query point $x_0 \sim g_0$ to its $m_1$-nearest-neighbor. Finally, the third error comes from the observational noise, where it is expected that the noise magnitude drops off by a factor of $\sqrt{m_2}$ after averaging. Hence, the optimal choice of the parameters $(m_1,m_2)$ aims to balance the above types of errors, with an additional constraint $\kappa m_2\le m_1$ to ensure that no single treatment covariate is matched to too many control covariates and therefore lead to a small variance. 

The full proof of Theorem \ref{thm.random_upper} is postponed to the appendix, but we remark that an appropriate maximal inequality is the key to handling general densities $g_0,g_1$ which may be close to zero. Specifically, let $f$ be any density function on $[0,1]^d$, and define the following \emph{minimal function} of $f$ as 
\begin{align}\label{eq:minimal_function}
m[f](x) \triangleq \inf_{0<r\le \sqrt{d}} \frac{1}{\text{Vol}_d( B(x;r)  )} \int_{B(x;r)} f(t)dt, 
\end{align}
where $B(x;r)$ denotes the ball centered at $x$ with radius $r$, $\text{Vol}_d$ is the $d$-dimensional volume, and $f(t)=0$ whenever $t\notin [0,1]^d$. It is clear that $m[f](x) \le f(x)$ for almost all $x\in [0,1]^d$, while the following lemma summarizes a key property of $m[f]$ which is very useful in the proof of Theorem \ref{thm.random_upper} and could be of independent interest. 
\begin{lemma}\label{lemma:minimal_inequality}
	For every $\lambda>0$, there exists a constant $C_d>0$ depending only on $d$ such that
	\begin{align*}
	\int_{[0,1]^d} f(x)e^{-\lambda \cdot m[f](x)}dx \le \begin{cases}
	\exp(-\lambda/eC_d) & \text{if } \lambda < eC_d, \\
	C_d / \lambda & \text{if } \lambda \ge eC_d. 
	\end{cases}
	\end{align*}
\end{lemma}

\subsection{Minimax Lower Bound}
The following minimax lower bound complements Theorem \ref{thm.random_upper} and shows the near minimax optimality (up to logarithmic factors) of the two-stage estimator in Algorithm \ref{alg:random}. 

\begin{theorem}\label{thm.random_lower}
	Assume that $\beta_\mu\le \beta_\tau \le 1$. Then under Gaussian noises and appropriate densities $g_0, g_1$ satisfying Assumption \ref{assu:covariate}, the minimax risk under the random design satisfies
	\begin{align*}
	R_{n,d,\beta_\mu,\beta_\tau,\sigma}^{\text{\rm random}}(\kappa) \ge \frac{1}{\mathsf{polylog}(n)}\left( \left(\frac{\kappa}{n^2}\right)^{\frac{1}{d(\beta_\mu^{-1} + \beta_\tau^{-1})}} + \left(\frac{\kappa\sigma^2}{n^2}\right)^{\frac{1}{2+d(\beta_\mu^{-1} + \beta_\tau^{-1})}} + \left(\frac{\kappa\sigma^2}{n}\right)^{\frac{\beta_\tau}{2\beta_\tau + d}} \right). 
	\end{align*}
\end{theorem}

The proof of Theorem \ref{thm.random_lower} is very involved and postponed to the appendix. Similar to the fixed design, the last error term comes from the classical nonparametric estimation of $\tau$, and establishing the first error term requires to solve a linear program in the same spirit to \eqref{eq.optimization}. However, as a regular grid is no longer available, constructing a near-optimal solution to \eqref{eq.optimization} becomes much more challenging. In addition, proving the second error term requires a hybrid approach of testing multiple hypotheses and constructing an exponential number of feasible solutions to \eqref{eq.optimization}. To this end, the following novel representation of the H\"{o}lder continuity condition is very useful throughout the lower bound argument. 
\begin{lemma}\label{lemma.holder}
	Let $\beta\le 1$, and $\{(x_i, y_i)\}_{i=1}^n$ be fixed covariate-value pairs on $[0,1]^d\times \mathbb{R}$. Then there exists some function $f\in \calH_d(\beta)$ (with radius $L$) with $f(x_i) = y_i$ for all $i\in [n]$ if and only if for any $i,j\in [n]$, 
	\begin{align*}
	|y_i - y_j| \le L\|x_i - x_j\|_2^\beta. 
	\end{align*}
\end{lemma}

\section{Numerical Experiments}\label{sec:experiment}
We illustrate the efficacy of the proposed estimator in Algorithm \ref{alg:random} via some numerical experiments. Specifically, we aim to show that the two main ingredients of Algorithm \ref{alg:random}, i.e. constructing pseudo-observations based on covariate matching and discarding observations with poor matching quality, are key to improved HTE estimation. We compare our estimator (which we call \emph{selected matching}) with the following three estimators: the \emph{full matching} estimator which never discards samples (i.e. $m_2=m_1$ always holds in Algorithm \ref{alg:random}), the \emph{kNN differencing} and \emph{kernel differencing} estimators which apply separate $k$-NN or kernel estimates to both baselines and then take the difference. Note that the idea of differencing-based estimators were essentially used in \cite{alaa2018limits,Kunzel4156}. 

The experimental setting is as follows. We choose the parameter values $n\in \{50,100,\cdots,1,000\}$, $d,\kappa\in \{1,\cdots,10\}$, and $\sigma = 2/\sqrt{n}$. Denote by $\phi_{\mu,\sigma^2}(x)$ the pdf of the normal random variable $\calN(\mu,\sigma^2)$, the baseline function and the HTE are chosen to be
\begin{align*}
\mu_0(x) = 2\phi_{0.1,0.15^2}(x) + 2.5\phi_{0.4,0.05^2}(x) + 4\phi_{0.8,0.1^2}(x), \quad \tau(x) = \phi_{0,1}(2x-1)
\end{align*}
for $d=1$ and $x\in [0,1]$, and for general dimensions, the variable $x$ is replaced by $\sqrt{d}(\bar{x}-0.5)+0.5$. An illustration of $\mu_0$ and $\tau$ for $d=1$ is plotted in Figure \ref{fig.HTE}, where $\tau$ is smoother and $\mu_0$ is more volatile. For each given $(n,d,\kappa,\sigma)$, we generate $n$ control covariates $X_1^0, \cdots, X_n^0$ following the i.i.d. density $g_0(x) = 2\kappa/(\kappa+1)$ when $x_1 \in [0,1/2]$ and $g_0(x) = 2/(\kappa+1)$ when $x_1\in (1/2,1]$. Similarly, the treatment covariates $X_1^1,\cdots,X_n^1$ are i.i.d. generated following the density $g_1(x) = 2-g_0(x)$, and the responses $Y_i^0, Y_i^1$ are defined in \eqref{eq.treatment_model} with i.i.d. $\calN(0,\sigma^2)$ noises. The algorithm parameters are determined by the optimal bias-variance tradeoffs in theory. Finally, the performance of HTE estimation is measured via the root mean squared error (RMSE) averaged over $100$ simulations. The source codes are available at \url{https://github.com/Mathegineer/NonparametricHTE}.

The experimental results are displayed in Figures \ref{fig.HTE} and \ref{fig.RMSE}, and we have the following observations. First, the differencing based estimators typically perform worse than the matching based estimators, especially at the places where the baseline is more volatile. Therefore, the differencing suffers more from the less smooth baselines. Second, although the \emph{full matching} estimator behaves similarly as \emph{selected matching} when $d=\kappa=1$, the performance of \emph{full matching} deteriorates significantly if either $d$ or $\kappa$ increases, and \emph{selected matching} is relatively more stable. This observation supports our intuition that the pairs with poor covariate matching quality should be discarded. 

\begin{figure}[htbp]
	\centering
	\begin{subfigure}[b]{0.32\textwidth}
		\includegraphics[width=\linewidth]{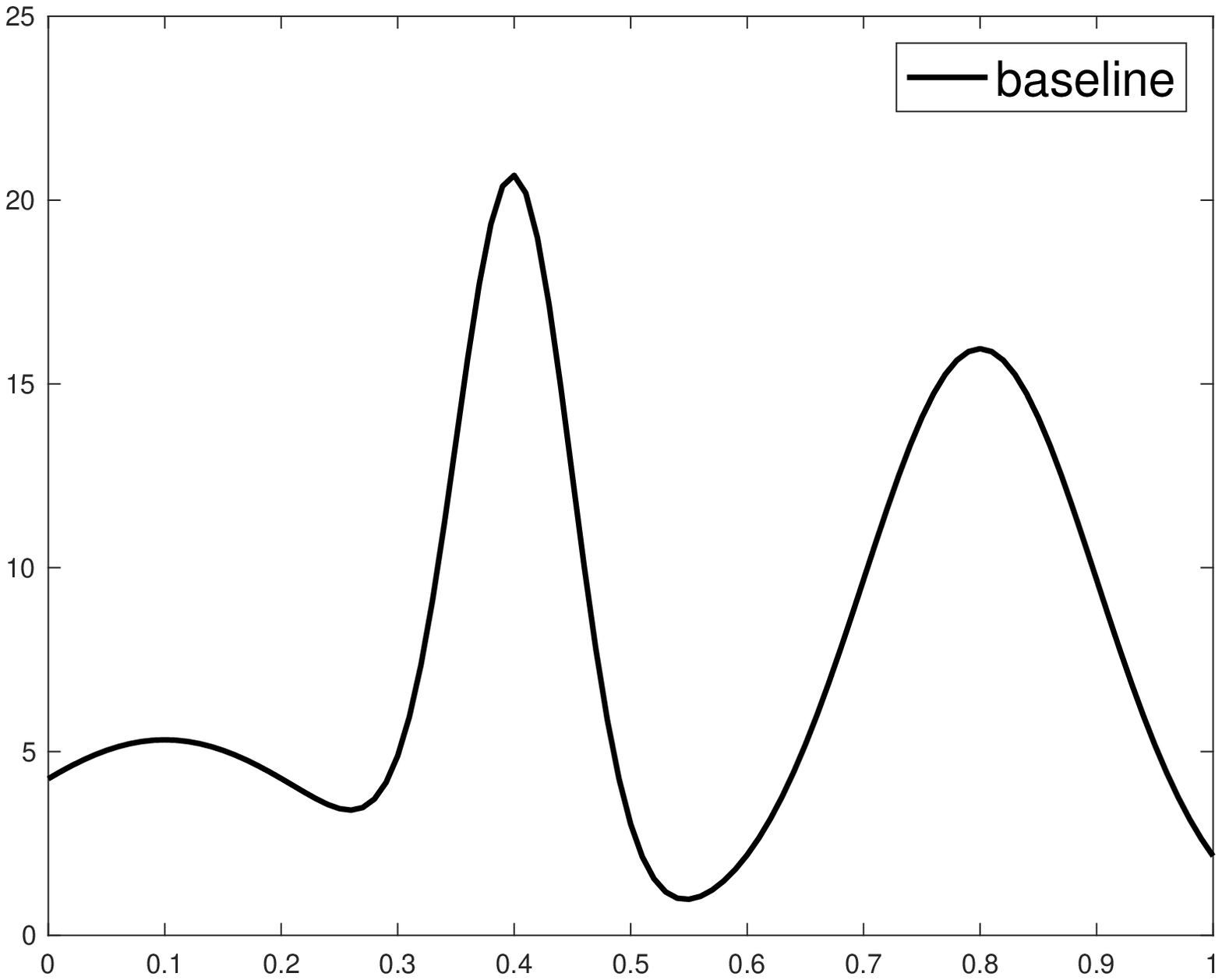}\caption{Baseline function $\mu_0$.}
	\end{subfigure}
	\begin{subfigure}[b]{0.32\textwidth}
		\includegraphics[width=\linewidth]{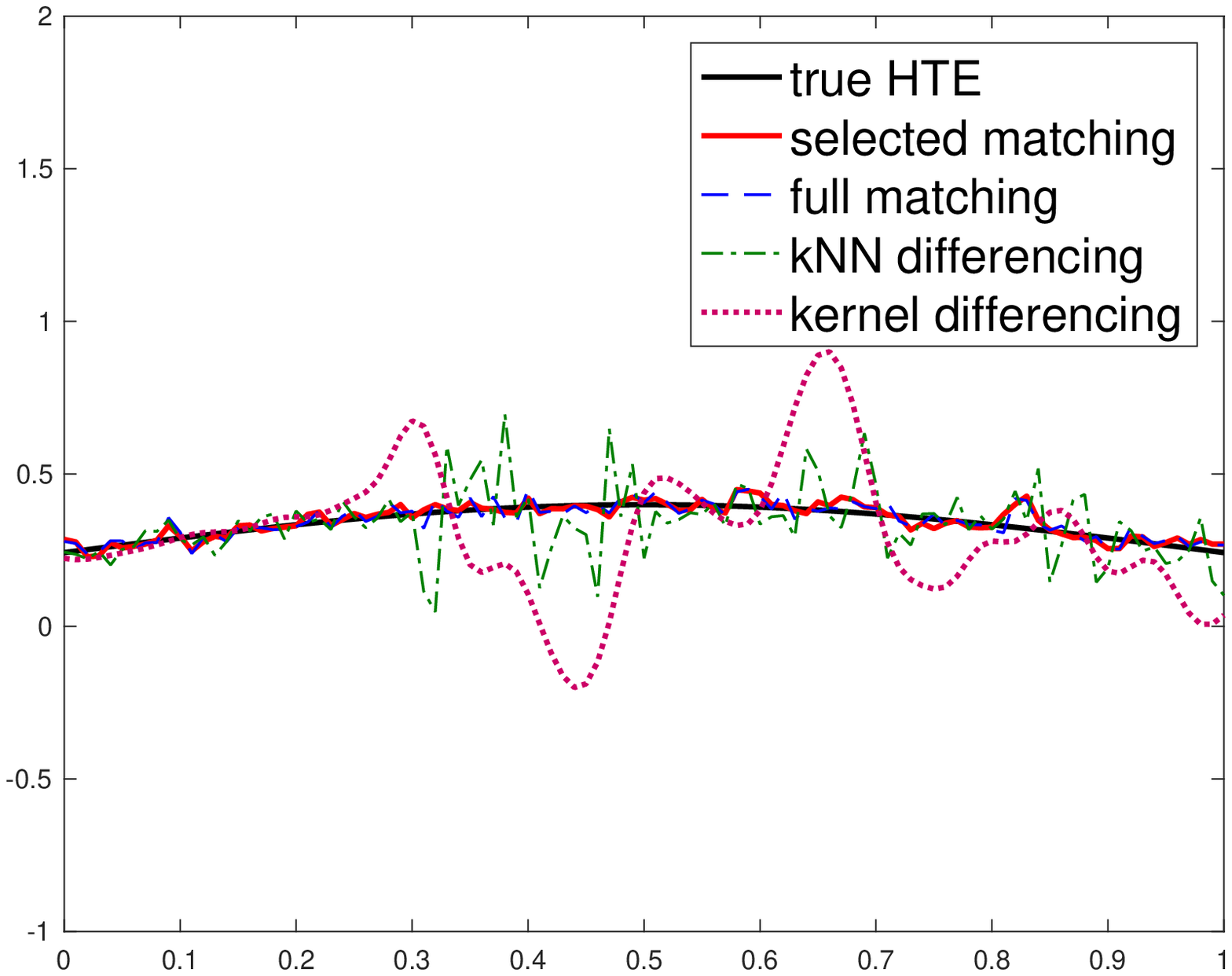}\caption{HTE $\tau$ with $\kappa=1$.}
	\end{subfigure}
	\begin{subfigure}[b]{0.32\textwidth}
		\includegraphics[width=\linewidth]{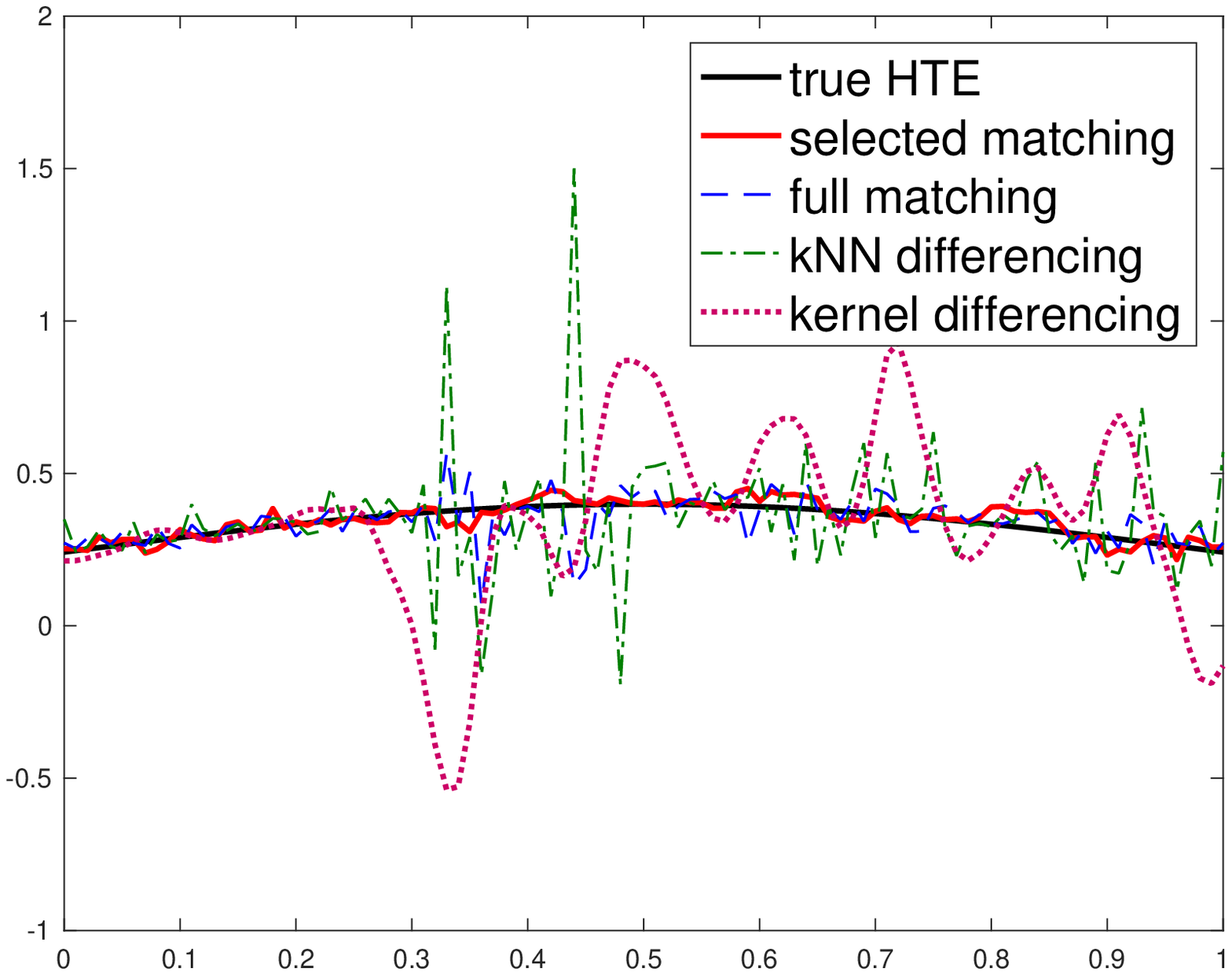}\caption{HTE $\tau$ with $\kappa=4$.}
	\end{subfigure}
	\caption{The baseline function $\mu_0$, HTE $\tau$, as well as the estimates.}\label{fig.HTE}
\end{figure}

\begin{figure}[htbp]
	\centering
	\begin{subfigure}[b]{0.32\textwidth}
		\includegraphics[width=\linewidth]{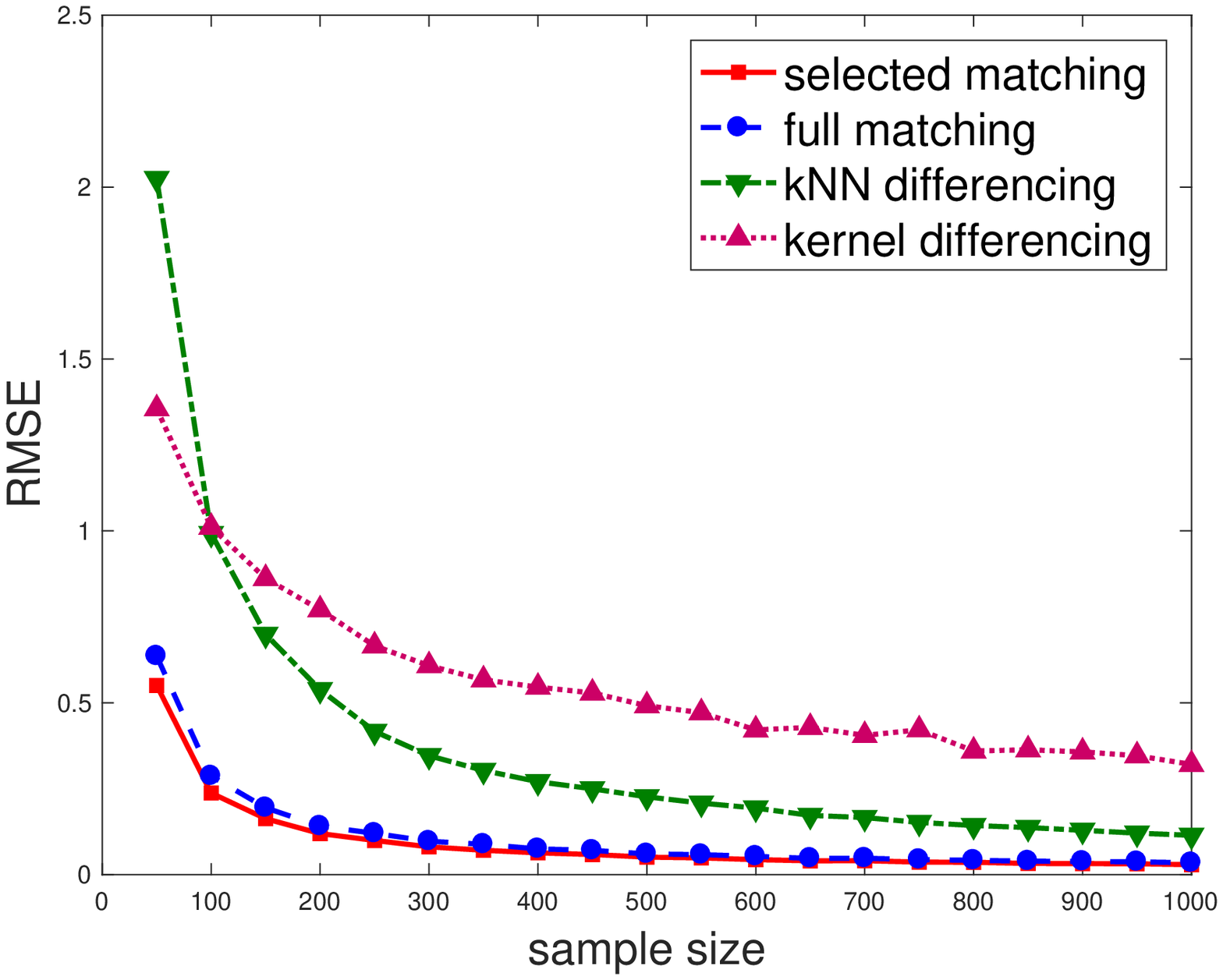}\caption{RMSE vs $n$.}
	\end{subfigure}
	\begin{subfigure}[b]{0.32\textwidth}
		\includegraphics[width=\linewidth]{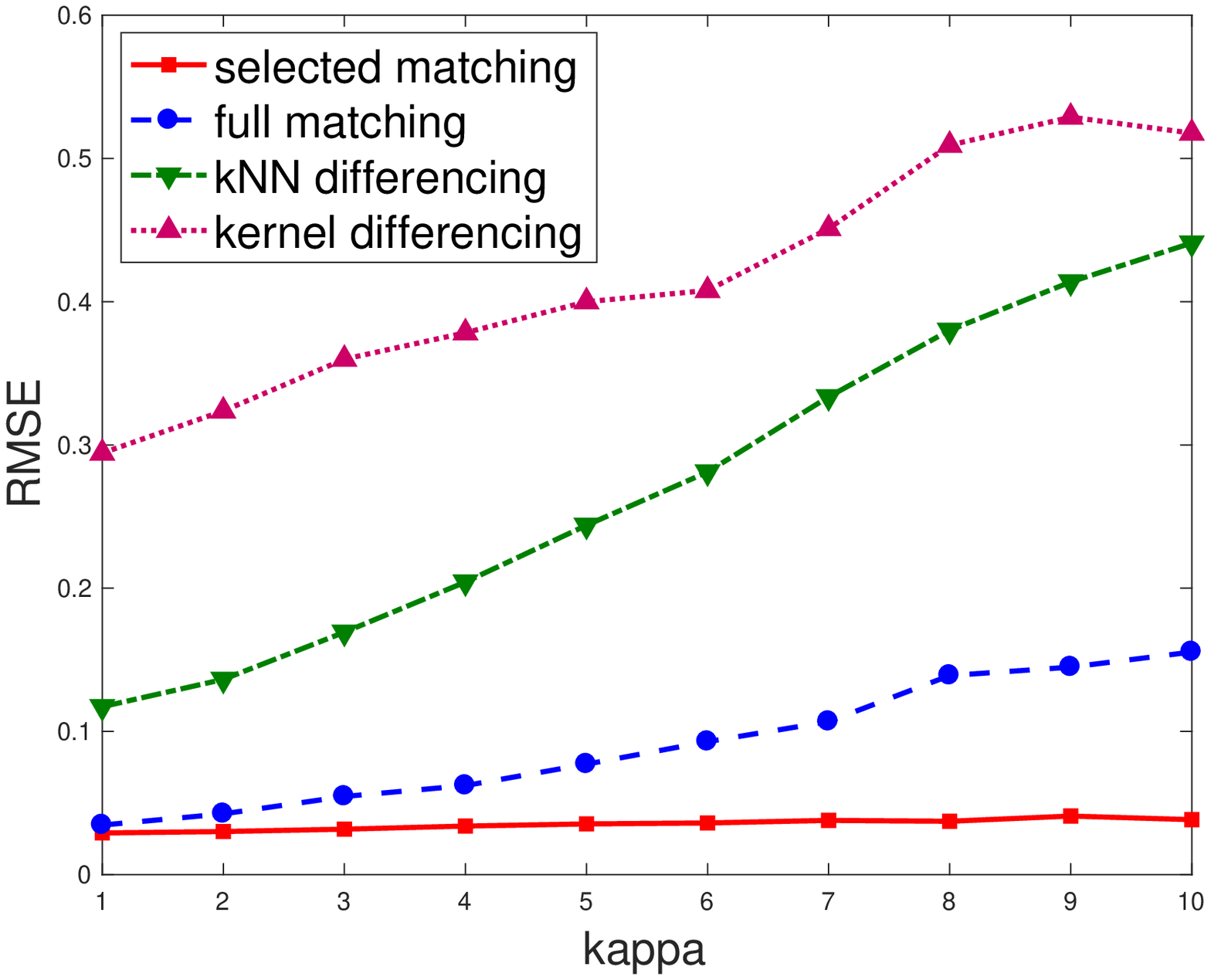}\caption{RMSE vs $\kappa$.}
	\end{subfigure}
	\begin{subfigure}[b]{0.32\textwidth}
		\includegraphics[width=\linewidth]{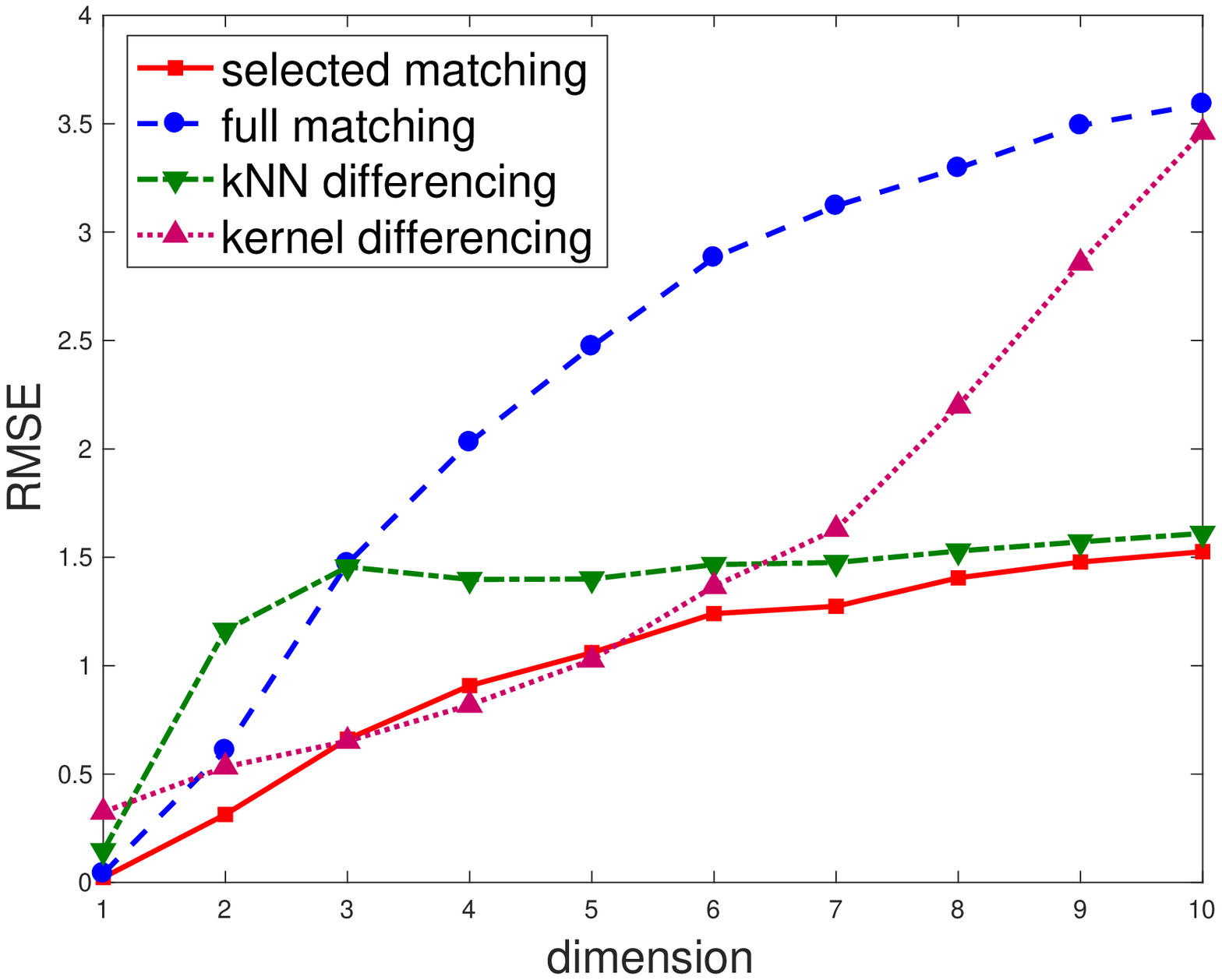}\caption{RMSE vs $d$.}
	\end{subfigure}
	\caption{The RMSE as a function of $n$, $\kappa$ or $d$ with default parameters $(n,\kappa,d)=(1000,1,1)$.}\label{fig.RMSE}
\end{figure}

\section{Further Discussions}\label{sec:discussion}
A major concern of this work is to capture higher order smoothness under the random design, which we remark is very challenging from both sides of achevability and lower bounds. For example, there are a number of open questions related to capturing higher order smoothness in nonparametric statistics, e.g. the adaptation to small densities \cite{patschkowski2016adaptation} despite the known asymptotic equivalence \cite{nussbaum1996asymptotic}, in the estimation of conditional variance \cite{richardson2014causal}, non-smooth functionals \cite{han2017optimal}, the analysis of nearest-neighbor methods \cite{jiao2018nearest}, local goodness-of-fit tests \cite{balakrishnan2019hypothesis}. Specializing to our setting, to capture the smoothness $\beta_{\mu}>1$ and reduce the matching bias, locally there need to be at least two treatment-control pairs mutually of small distance. However, instead of the $\Theta(n^{-2})$ minimum distance for $n$ uniform random variables in $[0,1]$, the counterpart for two pairs becomes at least $\Theta(n^{-4/3})$. Hence, we expect that there will be a phase transition from $\beta_\mu\le 1$ to $\beta_\mu>1$ in the estimation performance, and it is an outstanding open question to characterize this transition. 

As for the lower bound, the main difficulty of the extension to $\beta\ge 1$ lies in Lemma \ref{lemma.holder}. To the best of our knowledge, there is no simple and efficient criterion for the existence of a H\"{o}lder $\beta$-smooth function with specified values on general prescribed points. To illustrate the difficulty, assume that $\beta$ is an integer and consider the following natural generalization of the condition in Lemma \ref{lemma.holder}: for any $i_0,i_1,\cdots,i_{\beta}\in [n]$, it holds that $\beta! \cdot | f[x_{i_0}, x_{i_1} , \cdots, x_{i_\beta}]| \le L$, where $f[x_0,x_1,\cdots,x_m]$ is the divided difference of $f$ on points $x_0,x_1,\cdots,x_m$ \cite{de2005divided}: 
\begin{align*}
f[x_0,x_1,\cdots,x_m] \triangleq \sum_{i=0}^m f(x_i) \prod_{j\neq i} \frac{1}{x_i - x_j}. 
\end{align*}
Standard approximation theory shows that this condition is necessary, and it reduces to Lemma \ref{lemma.holder} for $\beta=1$. However, it is not sufficient when $\beta\ge 2$. For a counterexample, the value specification $f(-3) = f(-1) = 1, f(1) = f(3) = 0$ satisfies the divided difference condition with $\beta=2, L=1/4$, but no choice of $f(0)$ satisfies both $|f[-3,-1,0]| \le 1/8$ and $|f[0,1,3]|\le 1/8$. 

We also point out some other extensions. First, the current estimators rely on the knowledge of many parameters which are unlikely to be known in practice. For the smoothness parameters and the density ratio parameter, their impact factors through the bias and it is possible to apply Lepski's trick \cite{lepski1997optimal} or data-driven parameters \cite{fan1995data} to achieve full or partial adaptation. For other parameters such as the noise level, pre-estimation procedures such as \cite{donoho1995wavelet} are expected to be helpful. Second, in our work, we study the minimax risk after conditioning on both random realizations of the covariates and group assignments, meaning that the worst-case HTE can depend on both the above realizations. In contrast, \cite{nie2017quasi} only conditions on the random covariates and considers an expected risk taken with respect to the randomness in the group assignments, where the worst-case HTE cannot depend on the group assignments. The minimax risk in the latter setting remains unknown, and it is an outstanding question to compare the minimax rates in these settings. 

%\clearpage 
%\section*{Broader Impact}
%This work mainly provides theoretical tools and bounds for the HTE estimation in causal inference, as well as potentially useful practical insights such as the two-stage nearest-neighbors and throwing away observations with poor covariate matching quality. Hence, theorists and practitioners working on causal inference may potentially benefit from this work, and we expect no/little harm of this work to any researcher or applications. 


\appendix

%\begin{center}
%	\textbf{\large Supplement to ``Minimax Optimal Nonparametric Estimation of Heterogeneous Treatment Effects"}
%\end{center}
\section{Auxiliary Lemmas}
\begin{lemma}\label{lemma.weights}
	Fix any $t\in \NN$, $m\ge 1$ and $0\le \Delta\le 1/(2m)$. For any $i\in [m]$, let $x_0, \cdots, x_t$ be the $(t+1)$ nearest neighbors of $(i-1)/m$ in $\{\Delta, 1/m+\Delta, \cdots, 1-1/m+\Delta\}$ with increasing distance, and the weights $w_0,\cdots,w_t$ be the solution to
	\begin{align*}
	\sum_{j=0}^t w_j = 1, \qquad \sum_{j=0}^t w_j(x-x_j)^\ell = 0, \qquad \ell = 1,\cdots,t.
	\end{align*}
	Then there exists a constant $C>0$ depending only on $t$ such that
	\begin{align*}
	|w_0|\le C, \qquad |w_j| \le C\cdot m\Delta, \qquad j=1,2,\cdots,t. 
	\end{align*}
\end{lemma}

\begin{lemma}[Chernoff bound, Theorem 5.4 of \cite{mitzenmacher2005probability}]\label{lemma.binomial}
	For $X\sim \mathsf{B}(n,\lambda/n)$ and any $\delta\in (0,1]$, we have
	\begin{align*}
	\bP(X\ge (1+\delta)\lambda) &\le \left(\frac{e^{\delta}}{(1+\delta)^{1+\delta}} \right)^\lambda \le \exp\left(-\frac{\delta^2\lambda}{3}\right), \\
	\bP(X\le (1-\delta)\lambda) &\le \left(\frac{e^{-\delta}}{(1-\delta)^{1-\delta}} \right)^\lambda \le \exp\left(-\frac{\delta^2\lambda}{2}\right). 
	\end{align*}
\end{lemma}

\begin{lemma}[Generalized Hardy--Littlewood maximal inequality, a slight variant of Lemma 3 of \cite{jiao2018nearest}]\label{lemma.maximal}
	Let $\mu_1,\mu_2$ be two Borel measures that are finite on the bounded Borel sets of $\mathbb{R}^d$. Then, for all $t>0$ and any bounded Borel set $A \subseteq \mathbb{R}^d$,
	\begin{align*}
	\mu_1 \left( \left \{ {x}\in A: \sup_{\rho>0: B(x;\rho)\subseteq A} \left( \frac{\mu_2(B({x};\rho))}{\mu_1(B({x};\rho))} \right) \ge t \right \} \right ) \leq \frac{C_d}{t} \mu_2(A).
	\end{align*}
	Here $C_d>0$ is a constant that depends only on the dimension $d$. 
\end{lemma}

\section{Proof of Main Theorems} 
Throughout the proofs, we introduce the following notations: in addition to the asymptotic notation $a_n = \Theta(b_n)$, we also write $a_n = O(b_n)$ if $\limsup_{n\to\infty} a_n/b_n < \infty$, and $a_n = \Omega(b_n)$ if $b_n = O(a_n)$. Similarly, the notations $\widetilde{O}(\cdot)$, $\widetilde{\Omega}(\cdot)$ denote the respective meanings within logarithmic factors in $n$. 

\subsection{Proof of Theorem \ref{thm.minimax_fixed}}
The minimax risk of HTE estimation under the fixed design is a direct consequence of Theorems \ref{thm.fixed_upper} and \ref{thm.fixed_lower}. 
\subsubsection{Proof of Theorem \ref{thm.fixed_upper}}
By the definition of the pseudo-observation $\hat{Y}^1(x)$ and the potential outcome model \eqref{eq.treatment_model}, we have
\begin{align}\label{eq.difference_representation}
\hat{Y}^1(X_i^0) - Y_0^i = \tau(X_i^0) + b_\mu(X_i^0) + \tilde{\varepsilon}_i, 
\end{align}
with the new errors $\tilde{\varepsilon}_i$ defined as (with $t = \lfloor \beta_\mu \rfloor + 1$)
\begin{align*}
\tilde{\varepsilon}_i = \sum_{j\in [t]^d} w_j\cdot \varepsilon_j^1 - \varepsilon_i^0,
\end{align*}
and the matching bias $b_\mu(X_i^0)$ is
\begin{align*}
b_\mu(X_i^0) = \sum_{j\in [t]^d} w_j \cdot (\mu_0(x_j)+\tau(x_j)) - (\mu_0(X_i^0)+\tau(X_i^0)),
\end{align*}
where $w_j, x_j$ are the weights and the treatment covariates used to obtain pseudo-observations at $X_i^0$. By the smoothness property of $\mu_1 = \mu_0+\tau$ in Assumption \ref{assu:smoothness}, the Taylor expansion around $X_i^0$ gives
\begin{align*}
\mu_1(x_j) &= \sum_{\beta=0}^{\lfloor \beta_\mu \rfloor - 1} \sum_{\sum_{i=1}^d \beta_i = \beta} \frac{\partial^{\beta} \mu_1(X_i^0)}{\partial x_1^{\beta_1} \cdots \partial x_d^{\beta_d}}\prod_{k=1}^d \frac{(x_{j,k} - X_{i,k}^0)^{\beta_k}}{\beta_k!} \\
&\qquad + \sum_{\sum_{i=1}^d \beta_i = \lfloor \beta_\mu \rfloor} \frac{\partial^{\beta} \mu_1(\xi_j)}{\partial x_1^{\beta_1} \cdots \partial x_d^{\beta_d}}\prod_{k=1}^d \frac{(x_{j,k} - X_{i,k}^0)^{\beta_k}}{\beta_k!},
\end{align*}
where $\xi_j\in \bR^d$ is some point on the segment connecting $X_i^0$ and $x_j$. Hence, by the smoothness assumption of $\mu_1$, it holds that
\begin{align*}
\left| \mu_1(x_j) - \sum_{\beta=0}^{\lfloor \beta_\mu \rfloor} \sum_{\sum_{i=1}^d \beta_i = \beta} \frac{\partial^{\beta} \mu_1(X_i^0)}{\partial x_1^{\beta_1} \cdots \partial x_d^{\beta_d}}\prod_{k=1}^d \frac{(x_{j,k} - X_{i,k}^0)^{\beta_k}}{\beta_k!} \right| \le C\sum_{(\beta_1,\cdots,\beta_d) \in \Gamma}\prod_{k=1}^d |x_{j,k} - X_{i,k}^0|^{\beta_k},
\end{align*}
where $C>0$ is an absolute constant, and 
\begin{align*}
\Gamma = \left\{(\beta_1,\cdots,\beta_d): \beta_k\ge 0, \sum_{k=1}^d \beta_k = \beta_\mu, \text{at most one of }\beta_k \text{ is non-integer} \right\}
\end{align*}
is a finite set of indices. Now by the choice of each $w_j$ in \eqref{eq:weight}, we have
\begin{align}\label{eq.matching_bias}
|b_\mu(X_i^0)| \le C \sum_{(\beta_1,\cdots,\beta_d) \in \Gamma}\prod_{k=1}^d \left(\sum_{j=1}^t |w_{j,k}|\cdot |x_{j,k} - X_{i,k}^0|^{\beta_k} \right). 
\end{align}

Consider the term $|w_{j,k}|\cdot |x_{j,k} - X_{i,k}^0|^{\beta_k}$ for each index $j\in [t]$ and $k\in [d]$. If $x_{j,k}$ is the nearest neighbor of $X_{i,k}^0$ in the one-dimensional grid $\{\Delta_k,1/m+\Delta_k, \cdots, 1-1/m + \Delta_k\}$, then $|x_{j,k} - X_{i,k}^0|= \Delta_k$ and $|w_{j,k}|= O(1)$ by Lemma \ref{lemma.weights}. Otherwise, we have $|x_{j,k} - X_{i,k}^0|\le t/m$ and $|w_{j,k}|= O( m\Delta_k) $ again by Lemma \ref{lemma.weights}. Therefore, the following inequality always holds: 
\begin{align*}
|w_{j,k}|\cdot |x_{j,k} - X_{i,k}^0|^{\beta_k} = O\left( m^{-\beta_k}(m\Delta_k)^{\beta_k\wedge 1}\right) = O\left( m^{-\beta_k}(m\|\Delta\|_\infty)^{\beta_k\wedge 1} \right).  
\end{align*}
Consequently, for any $i\in [n]$ it holds that
\begin{align*}
|b_\mu(X_i^0)| = O\left( Ct^d\cdot |\Gamma|\max_{(\beta_1,\cdots,\beta_d)\in \Gamma} \left( \prod_{k=1}^d m^{-\beta_k}(m\|\Delta\|_\infty)^{\beta_k\wedge 1} \right) \right) = O\left( m^{-\beta_\mu}(m\|\Delta\|_\infty)^{\beta_\mu\wedge 1} \right),
\end{align*}
which gives the claimed matching bias upper bound in the error decomposition of Theorem \ref{thm.fixed_upper}. 

The rest of the proof follows from the classic analysis of the Nadaraya-Watson estimator applied to \eqref{eq.difference_representation} (see, e.g. \cite{tsybakov2008introduction}), where the only difference is that the errors $\tilde{\varepsilon}_i$ are weakly dependent instead of being mutually independent, i.e. each $\tilde{\varepsilon}_i$ depends only on finitely many $\tilde{\varepsilon}_j$'s. 

\subsubsection{Proof of Theorem \ref{thm.fixed_lower}}
To prove the lower bound $\Omega((\sigma^2/n)^{\beta_\tau/(2\beta_\tau+d)})$, consider $\mu_0 \equiv 0$. In this scenario, the outcomes of the control group are completely non-informative for estimating $\tau$, and the treatment outcome model \eqref{eq.treatment_model} is precisely the nonparametric regression model for the $\beta_\tau$-H\"{o}lder smooth function $\tau$. Then by the standard minimax lower bound for nonparametric regression (see, e.g. \cite{tsybakov2008introduction}), the lower bound $\Omega((\sigma^2/n)^{\beta_\tau/(2\beta_\tau+d)})$ is immediate. 

Next we prove the lower bound $\Omega(n^{-\beta_\mu/d}(n^{1/d}\|\Delta\|_\infty)^{\beta_\mu \wedge 1})$, which corresponds to the matching bias. We first construct a feasible solution $(\mu_0, \tau)$ to the optimization problem \eqref{eq.optimization}. Without loss of generality we assume that $\Delta_1 = \|\Delta\|_\infty$. Construct the following function $g$ on $[0,1]^d$: for $x=(x_1,\cdots,x_d)\in [0,1]^d$, 
\begin{align*}
g(x) = c\cdot [x_1(1-x_1)]^{\beta_\mu \wedge 1}
\end{align*}
with some small constant $c>0$. We claim that $g$ is $\beta$-H\"{o}lder smooth on $[0,1]^d$ given a sufficiently small $c>0$. In fact, if $\beta_\mu\le 1$, for all $x_1, x_1'\ge 0$ simple algebra gives $|(x_1')^{\beta_\mu} - x_1^{\beta_\mu}| \le |x_1' - x_1|^{\beta_\mu}$, and therefore
\begin{align*}
|g(x) - g(x')| &= c\cdot \left|[x_1(1-x_1)]^{\beta_\mu} - [x_1'(1-x_1')]^{\beta_\mu} \right| \\
&\le c\cdot \left((1-x_1)^{\beta_\mu}|x_1^{\beta_\mu} - (x_1')^{\beta_\mu} | + (x_1')^{\beta_\mu}|(1-x_1)^{\beta_\mu} - (1-x_1')^{\beta_\mu}| \right) \\
&\le c\cdot \left(1\cdot |x_1 - x_1'|^{\beta_\mu} + 1\cdot |x_1 - x_1'|^{\beta_\mu}\right) \\
&= 2c |x_1 - x_1'|^{\beta_\mu} \le 2c \|x-x'\|_2^{\beta_\mu}. 
\end{align*}
If $\beta_\mu>1$, the function $g(x) = cx_1(1-x_1)$ is a smooth function in $x$ and therefore $\beta_\mu$-H\"{o}lder smooth as well. 

Recall that $(0,1]^d = \cup_{i=1}^n (X_i^0 + (0,m^{-1}]^d)$ under the fixed design. Consider the following construction of $\mu_0$: 
\begin{align*}
\mu_0(x) = h(x)\cdot \frac{1}{m^{\beta_\mu}}\sum_{i=1}^n g\left(m(x - X_i^0) \right) \1( m(x-X_i^1)\in (0,1]^d ),
\end{align*}
where the function $g$ is defined above, and $h$ is an arbitrary smooth function on $\bR^d$ supported on $[0,1]^d$ with $h(x) \equiv 1$ for $x\in [1/4,3/4]^d$. In other words, the baseline function $\mu_0$ is a dilation of the reference function $g$ with proper scaling to preserve $\beta_\mu$-smoothness, and the function $h$ preserves the value of $\mu_0$ in the center of $[0,1]^d$ while helps $\mu_0$ connect to zero smoothly at the boundary of $[0,1]^d$. Since multiplying a bounded smooth function does not decrease the smoothness parameter, it is straightforward to verify that $\mu_0\in \calH_d(\beta_\mu)$ fulfills Assumption \ref{assu:smoothness}. 

We also construct the HTE $\tau$ as $\tau(x) = -m^{-\beta_\mu}g(m\Delta)\cdot h(x)$, where $h$ is the same smooth function used in the definition of $\mu_0$. By the smoothness of $h$, it is clear that $\tau\in \calH_d(\beta_\tau)$ for any $\beta_\tau >0$. Moreover, the pair $(\mu_0, \tau)$ is a feasible solution to \eqref{eq.optimization}: the smoothness conditions have already been verified, and 
\begin{align*}
\mu_0(X_i^0) &= h(X_i^0)\cdot m^{-\beta_\mu}g(0) = 0, \\
\mu_0(X_i^1) + \tau(X_i^1) &= h(X_i^1)\cdot m^{-\beta_\mu}g(m\Delta) - m^{-\beta_\mu}g(m\Delta)\cdot h(X_i^1) = 0. 
\end{align*}
Also, the construction of $g$ ensures that $|g(m\Delta)| = \Omega((m\|\Delta\|_\infty)^{\beta_\mu \wedge 1})$, and the fact $h(x)\equiv 1$ for $x\in [1/4,3/4]^d$ gives $\|h\|_1 = \Omega(1)$. Consequently, we have
\begin{align*}
\|\tau\|_1 = m^{-\beta_\mu}|g(m\Delta)|\cdot \|h\|_1 = \Omega\left(n^{-\frac{\beta_\mu}{d}} (n^{\frac{1}{d}}\|\Delta\|_\infty)^{\beta_\mu \wedge 1} \right). 
\end{align*}

Finally, note that the distributions of the outcomes \eqref{eq.treatment_model} in two groups are the same under the above construction of $(\mu_0, \tau)$ and the na\"{i}ve construction $(\mu_0' \equiv 0, \tau' \equiv 0)$, the standard two-point method yields to a minimax lower bound $\|\tau - \tau'\|_2$, which is exactly the desired matching bias. 

\subsection{Proof of Theorem \ref{thm.minimax_random}}
The minimax risk of HTE estimation under the random design is a direct consequence of Theorems \ref{thm.random_upper} and \ref{thm.random_lower}. 

\subsubsection{Proof of Theorem \ref{thm.random_upper}}
For each $x_0\in [0,1]^d$, the construction of the two-stage nearest-neighbor based estimator $\hat{\tau}(x_0)$ in Algorithm \ref{alg:random} gives
\begin{align*}
\hat{\tau}(x_0) - \tau(x_0) &= \frac{1}{m_2}\sum_{i\in I_2} (\mu_0(X_{j(i)}^1) - \mu_0(X_i^0) ) + \frac{1}{m_2}\sum_{i\in I_2} (\tau(X_{j(i)}^1) - \tau(x_0)) + \frac{1}{m_2}\sum_{i\in I_2} (\varepsilon_{j(i)}^1 - \varepsilon_i^0) \\
&\triangleq b_\mu(x_0) + b_\tau(x_0) + s(x_0), 
\end{align*}
where the above terms correspond to the matching bias incurred by $\mu$, estimation bias incurred by $\tau$, and the stochastic error incurred by the noises, respectively. As the final estimation error is measured under $L_1(g_0)$, it suffices to upper bound the quantites $\bE|b_{\mu}(X)|$, $\bE|b_\tau(X)|$, and $\bE|s(X)|$ for a fresh test covariate $X\sim g_0$, respectively. 

We first upper bound the quantity $\bE|b_\mu(X)|$. As in the description of the Algorithm \ref{alg:random}, for $i\in I_1$, let $d_i = \min_{j\in [n]} \|X_i^0 - X_j^1\|_2$ be the minimum Euclidean distance between the control covariate $X_i^0$ and all possible treatment covariates, and $D_1$ be the random variable representing the $m_2$-th smallest value of $(d_i)_{i\in I_1}$. Then by the smoothness condition of $\mu_0$, it is clear that 
\begin{align}\label{eq:b_mu}
\bE|b_\mu(X)| \lesssim \bE[D_1^{\beta_\mu}] \le (\bE[D_1])^{\beta_\mu}. 
\end{align}
Hence, it remains to upper bound the expectation $\bE[D_1]$. First, for any $x\in [0,1]^d, j\in [n]$ and $t\in (0,\sqrt{d})$, we have
\begin{align*}
\bP\left(\|X_j^1 - x\|_2 \ge t \right) &= \int_{y\in [0,1]^d: y\notin B(x;t)} g_1(y)dy \le 1 - \gamma_dt^d\cdot m[g_1](x), 
\end{align*}
by the definition of the minimal function defined in \eqref{eq:minimal_function}, where $\gamma_d$ is the volume of the $d$-dimensional unit ball. Consequently, for each $i\in I_1$ and $0<t\lesssim n^{-1/d}$, we have
\begin{align*}
\bP(d_i \ge t) &= \bE\left[ \bP\left(\|X_1^1 - X_i^0\|_2 \ge t  \mid X_i^0\right)^n \right] \\
&\le \bE\left[\left(1 - \gamma_d t^d\cdot m[g_1](X_i^0) \right)^n \right] \\
&\stepa{\le} \bE\left[\exp\left( - n\gamma_d t^d\cdot m[g_1](X_i^0)\right) \right] \\
&= \int_{[0,1]^d} g_0(x)\exp\left( - {n\gamma_d t^d}m[g_1](x)\right) dx \\
&\stepb{\le} \int_{[0,1]^d} g_0(x)\exp\left( - \frac{n\gamma_d t^d}{\kappa}m[g_0](x)\right) dx \\
&\stepc{\le} \exp\left( -\frac{c_dnt^d}{\kappa} \right), 
\end{align*}
where (a) is due to the elementary inequality $1-x\le \exp(-x)$, (b) follows from the upper bound on the density ratio, and (c) is due to Lemma \ref{lemma:minimal_inequality} and the assumption $t\lesssim n^{-1/d}$, with some absolute constant $c_d>0$. Now by the mutual independence of control and treatment covariates, the random variable $D_1$ is simply the $m_2$-th order statistic of $(d_i)_{i\in I_1}$, and therefore
\begin{align*}
\bP(D_1 \ge t) &= \bP\left(\mathsf{B}\left(m_1, \bP(d_1\ge t) \right) \ge m_1 - m_2+1\right) \\
&\le \bP\left(\mathsf{B}\left(m_1, \exp(-c_dnt^d/\kappa) \right) \ge m_1 - m_2+1\right) \\
&= \bP\left(\mathsf{B}\left(m_1, 1- \exp(-c_dnt^d/\kappa) \right) < m_2\right). 
\end{align*}
Hence, for $t\ge C_d(\kappa m_2/nm_1)^{1/d}$ with a large enough constant $C_d>0$, we have
\begin{align*}
1 - \exp\left(-c_dnt^d/\kappa \right) \le \frac{m_2}{2m_1}, 
\end{align*}
and therefore Chernoff bound (Lemma \ref{lemma.binomial}) gives the upper bound $\bP(D_1\ge t)=\exp(-\Omega(m_2))$, which is smaller than $O(n^{-1})$ by the assumption that $m_2$ is at least logarithmic in $n$. Consequently, 
\begin{align*}
\bE[D_1] = \int_0^{\sqrt{d}} \bP(D_1 \ge t)dt \le C_d\left(\frac{\kappa m_2}{nm_1}\right)^{\frac{1}{d}} + \sqrt{d}\cdot \frac{1}{n} \lesssim \left(\frac{\kappa m_2}{nm_1}\right)^{\frac{1}{d}},
\end{align*}
and consequently \eqref{eq:b_mu} leads to 
\begin{align}\label{eq:b_mu_final}
\bE|b_\mu(X)| \lesssim \left(\frac{\kappa m_2}{nm_1}\right)^{\frac{\beta_\mu}{d}}. 
\end{align}

We then upper bound the quantity $\bE|b_\tau(X)|$. Let $D_2$ be the non-negative random variable representing the Euclidean distance between the query point $X$ and its $m_1$-nearest neighbor in the control group, then standard nearest neighbor analysis (see, e.g. \cite[Theorem 2.4]{biau2015lectures}) gives $\bE[D_2]\lesssim (m_1/n)^{1/d}$. Consequently, 
\begin{align}\label{eq:b_tau_final}
\bE|b_\tau(X)| &\le \frac{1}{m_2}\sum_{i\in I_2}\bE\left( |\tau(X_{j(i)}^1) - \tau(X_i^0)| + |\tau(X_i^0) - \tau(X)| \right) \nonumber\\
&\lesssim \bE[D_1^{\beta_\tau}] + \bE[D_2^{\beta_\tau}] \nonumber \\
&\le (\bE[D_1])^{\beta_\tau} + (\bE[D_2])^{\beta_\tau} \nonumber\\
&\lesssim \left(\frac{\kappa m_2}{nm_1}\right)^{\frac{\beta_\tau}{d}} + \left(\frac{m_1}{n}\right)^{\frac{\beta_\tau}{d}} \nonumber \\
&\lesssim \left(\frac{m_1}{n}\right)^{\frac{\beta_\tau}{d}}, 
\end{align}
where the last inequality follows from the assumption that $\kappa m_2\le m_1$. 

Finally we deal with the stochastic error $\bE|s(X)|$. Clearly, due to the mutual independence of the noises in Assumption \ref{assu:noise}, we have
\begin{align*}
\bE[s(X)^2] \lesssim \frac{\sigma^2}{m_2^2}\bE\left[\sum_{j\in [n]} S_j^2 \right], 
\end{align*}
where $S_j = \sum_{i\in I_2} \1(i\to j)$ is the total number of times $X_j^1$ is chosen to be the nearest neighbor, and $i\to j$ denotes that $X_j^1$ is the nearest neighbor of $X_i^0$ among the treatment covariates. Consequently, 
\begin{align}\label{eq:S_j}
\bE\left[\sum_{j\in [n]} S_j^2 \right] &= \bE\left[\sum_{i\in I_2} \sum_{i'\in I_2} \sum_{j\in [n]} \1(i\to j)\1(i'\to j) \right] \nonumber\\
&=  \bE\left[\sum_{i\in I_2} \sum_{j\in [n]} \1(i\to j) \right] + \bE\left[\sum_{i\neq i'\in I_2} \sum_{j\in [n]} \1(i\to j) \1(i'\to j)\right] \nonumber\\
&= m_2 + \bE\left[\sum_{i\neq i'\in I_2} \sum_{j\in [n]} \1(i\to j) \1(i'\to j)\right], 
\end{align}
where the last identity follows from $\sum_{j\in [n]} \1(i\to j) = 1$. To deal with the second term, first assume that the control and treatment covariates have the same distribution. Then
\begin{align*}
\bE\left[\sum_{i\neq i'\in I_2} \sum_{j\in [n]} \1(i\to j) \1(i'\to j)\right] \le  \bE\left[\sum_{i\in I_2} \sum_{i'\in [n], i'\neq i} \sum_{j\in [n]} \1(i\to j) \1(i'\to j)\right],
\end{align*}
and we only need to upper bound each quantity $\bE[\1(i\to j)\1(i'\to j)]$ for $i\in I_2, i'\in [n]\backslash \{i\}$, and $j\in [n]$. By the i.i.d. assumption, the covariates $X_1^1,\cdots,X_n^1,X_{i'}^0$ are i.i.d. conditioning on $X_i^0$. This observation motivates us to consider the following problem: let $Y_0=x\in [0,1]^d$ be any fixed point, and $Y_1,\cdots,Y_{n+1}$ be i.i.d. distributed as $X_1^1$. For $i\in [n+1]\cup\{0\}$ and $j\in [n+1]$ with $i\neq j$, let $d(i,j)\in [n]$ be the integer $d$ such that $Y_j$ is the $d$-th nearest neighbor of $Y_i$ among $\{Y_1,\cdots,Y_{i-1},Y_{i+1},\cdots,Y_{n+1}\}$. Consequently, identifying $X_i^0, X_{i'}^0, X_j^1$ as $Y_0, Y_{n+1}$, and $Y_{1}$ respectively, the i.i.d. assumption gives
\begin{align*}
\bE\left[\1(i\to j)\1(i'\to j) \mid X_i^0 = x\right] \le \bE[\1(d(0,1)\le 2)\1(d(n+1,1)=1)]. 
\end{align*}
To upper bound the RHS, note that the exchangeability of $(X_1,\cdots,X_{n+1})$ yields
\begin{align*}
\bE[\1(d(0,1)\le 2)\1(d(n+1,1)=1)] &= \frac{1}{n(n+1)}\sum_{i\neq j\in [n+1]} \bE[\1(d(0,i)\le 2)\1(d(j,i)=1)] \\
&= \frac{1}{n(n+1)}\bE\left[\sum_{i=1}^{n+1}\1(d(0,i)\le 2)\sum_{j\neq i} \1(d(j,i)=1)\right] \\
&\stepa{\le} \frac{1}{n(n+1)}\bE\left[\sum_{i=1}^{n+1}\1(d(0,i)\le 2)\cdot c_d\right] \\
&\stepb{=} \frac{2c_d}{n(n+1)}, 
\end{align*}
where (a) follows from the deterministic inequality $\sum_{j\neq i} \1(d(j,i)=1) \le c_d$ in \cite[Lemma C.1]{gao2017estimating}, i.e. each $X_j$ can only be the nearest neighbor of at most $c_d$ points, with constant $c_d$ depending only on $d$. The identity (b) trivially follows from the definition of nearest-neighbors and $d(0,i)$. Consequently, combining the above displays and \eqref{eq:S_j}, we arrive at 
\begin{align*}
\bE\left[\sum_{j\in [n]} S_j^2 \right]  \le (1+2c_d)m_2,
\end{align*}
and therefore 
\begin{align}\label{eq:s_final}
\bE[|s(X)|] \le \sqrt{\bE[s(X)^2]} \lesssim \frac{\sigma}{\sqrt{m_2}}. 
\end{align}
In the general case, as $I_2$ is essentially a uniformly random subset of $I_1$, we have
\begin{align*}
\bE\left[\sum_{i\neq i'\in I_2} \sum_{j\in [n]} \1(i\to j) \1(i'\to j)\right] &\lesssim \left(\frac{m_2}{m_1}\right)^2\bE\left[\sum_{i\neq i'\in I_1} \sum_{j\in [n]} \1(i\to j) \1(i'\to j)\right]\\
&\le \left(\frac{m_2}{m_1}\right)^2\bE\left[\sum_{i\in I_1} \sum_{i'\in [n], i'\neq i} \sum_{j\in [n]} \1(i\to j) \1(i'\to j)\right]. 
\end{align*}
Moreover, by the likelihood ratio assumption, a simple change of measure also gives
\begin{align*}
\bE\left[\1(i\to j)\1(i'\to j) \mid X_i^0 = x\right] \le \kappa\cdot \bE[\1(d(0,1)\le 2)\1(d(n+1,1)=1)].
\end{align*}
Consequently, as $m_1\ge \kappa m_2$, we again arrive at the same upper bound \eqref{eq:s_final}, and the claimed Theorem \ref{thm.random_upper} follows from \eqref{eq:b_mu_final}, \eqref{eq:b_tau_final}, and \eqref{eq:s_final}. 

\subsubsection{Proof of Theorem \ref{thm.random_lower}}
We first show that to prove the lower bound of the minimax rate, it suffices to consider the case where both $g_0$ and $g_1$ are uniform over $[0,1]^d$, and there are $n$ and $m=n/\kappa$ (assumed to be an integer as $\kappa\le n$) observations for the control and treatment groups, respectively. To see this, note that we may choose $g_0$ to be the uniform distribution on $[0,1]^d$, and $g_1$ be $1/\kappa$ in a smaller cube and $\kappa$ otherwise. Note that the smaller cube has the volume $\kappa/(\kappa+1)$, which is at least $\Omega(1)$. Besides, the conditional distribution of $g_1$ restricted to the smaller cube is again uniform, and with high probability there are $O(n)$ control observations and $O(n/\kappa)$ treatment observations in the smaller cube. Consequently, restricting everything including the minimax risk to the smaller cube and using a proper scaling of constant factor, the desired reduction holds for the purpose of proving the lower bound. Furthermore, in this case the $L_1(g_0)$ norm coincides with the usual $L_1$ norm. 

As in the proof of Theorem \ref{thm.fixed_lower}, after choosing $\mu\equiv 0$, the lower bound $\Omega((\sigma^2/m)^{\beta_\tau /(2\beta_\tau+d) })$ follows directly from the known results in nonparametric estimation. Hence, using the relationship $m=n/\kappa$, the last error term of Theorem \ref{thm.random_lower} is established. For the other two errors, we will need to construct appropriate functions in $\calH_d(\beta_\mu)$ and $\calH_d(\beta_\tau)$, and by multiplying a smooth function supported on $[0,1]^d$ (as in the proof of Theorem \ref{thm.fixed_lower}) we will not consider the boundary effects in the construction below. 

First we prove the lower bound $\widetilde{\Omega}((\kappa/n^2)^{1/d(\beta_\mu^{-1} + \beta_\tau^{-1})})$ via providing another feasible solution to the optimization problem \eqref{eq.optimization}. Specifically, let $\mu_0, \tau$ be the H\"{o}lder smooth functions with their respective smoothness parameters $\beta_\mu, \beta_\tau$ assured by Lemma \ref{lemma.holder} based on the following value specifications: 
\begin{align*}
\mu_0(X_i^0) = 0, \qquad \mu_0(X_i^1) = -\tau(X_i^1) = c\cdot \min_{j\in [m]}\left(\|X_i^1 - X_j^1\|_2^{\beta_\tau} + \min_{k \in [n]} \|X_j^1 - X_k^0\|_2^{\beta_\mu} \right), 
\end{align*}
for all $i\in [m]$, where $c>0$ is a small constant. However, before applying Lemma \ref{lemma.holder}, we need to show that the conditions of Lemma \ref{lemma.holder} hold for the above value specifications. 

For the HTE function $\tau$, using the Lipschitz property of the minimum $|\min_k a_k - \min_k b_k|\le \max_k |a_k - b_k|$, we have
\begin{align*}
&|\tau(X_i^1) - \tau(X_{i'}^1)| \\
&\le c\cdot \max_{j\in [m]} \left| \|X_i^1 - X_j^1\|_2^{\beta_\tau} + \min_{k \in [n]} \|X_j^1 - X_k^0\|_2^{\beta_\mu} - \left(\|X_{i'}^1 - X_j^1\|_2^{\beta_\tau} + \min_{k \in [n]} \|X_j^1 - X_k^0\|_2^{\beta_\mu} \right) \right| \\
&= c\cdot \max_{j\in [m]} \left| \|X_i^1 - X_j^1\|_2^{\beta_\tau} - \|X_{i'}^1 - X_j^1\|_2^{\beta_\tau} \right| \\
&\le c\cdot \|X_i^1 - X_{i'}^1\|_2^{\beta_\tau},
\end{align*}
where the last step follows from $| |a|^{\beta_\tau} - |b|^{\beta_\tau} |\le |a-b|^{\beta_\tau}$ for all $\beta_\tau \in [0,1]$. Hence, for constant $c>0$ sufficiently small, the condition of Lemma \ref{lemma.holder} holds for $\tau$. 

For the baseline function $\mu_0$, using the same analysis we have $|\mu_0(X_i^1) - \mu_0(X_{i'}^1)| = O(\|X_i^1 - X_{i'}^1\|_2^{\beta_\tau}) = O(\|X_i^1 - X_{i'}^1\|_2^{\beta_\mu})$ for the treatment-treatment pairs. However, for $\mu_0$ we also need to verify the condition of Lemma \ref{lemma.holder} for control-control and treatment-control pairs. The the control-control pairs are easy to verify: $|\mu_0(X_i^0) - \mu_0(X_{i'}^0)| = 0$ always holds. As for the treatment-control pairs, we have
\begin{align*}
|\mu_0(X_i^1) - \mu_1(X_{i'}^0)| &= |\mu_0(X_i^1)| \\
&= c\cdot \min_{j\in [m]}\left(\|X_i^1 - X_j^1\|_2^{\beta_\tau} + \min_{k \in [n]} \|X_j^1 - X_k^0\|_2^{\beta_\mu} \right) \\
&\le c\cdot \min_{k\in [n]} \|X_i^1 - X_k^0\|_2^{\beta_\mu} \\
&\le c\cdot \|X_i^1 - X_{i'}^0\|_{2}^{\beta_\mu},
\end{align*}
as desired. Hence, both functions $\mu_0$ and $\tau$ can be extended to H\"{o}lder smooth functions on $[0,1]^d$ with desired smoothness parameters, and $(\mu_0, \tau)$ is a feasible solution to the optimization problem \eqref{eq.optimization}. 

Next we provide a lower bound of $\|\tau\|_1$ for the above construction of $\tau$. Choose a bandwidth
\begin{align*}
h_1 = \widetilde{\Theta}\left((mn)^{-\frac{\beta_\mu}{d(\beta_\tau+\beta_\mu)}} \right)
\end{align*}
and define
\begin{align*}
I = \left\{i\in [m]: \min_{j\in [n]}\|X_i^1 - X_j^0\|_2 \le \frac{1}{h_1}\left(\frac{1}{mn\log n}\right)^{\frac{1}{d}} \right\}. 
\end{align*}
The next lemma shows that the cardinality of $I$ is upper bounded by $1/(3h^d)$ with high probability. 
\begin{lemma}\label{lemma:min_dist}
	If $h_1 \le c(\log n)^{-2/d}$ with constant $c>0$ small enough, then with probability at least $1-n^{-3}$ we have $|I| \le 1/(3h^d)$ for sufficiently large $n$. 
\end{lemma}
\begin{proof}
	As the control covariates follow the uniform distribution, for any $x\in [0,1]^d$ we have
	\begin{align*}
	\bP\left( \|x - X_j^0\|_2 \ge \frac{1}{h_1(mn\log n)^{1/d}} \right) \ge 1 - v_d\left(\frac{1}{h_1(mn\log n)^{1/d}} \right)^d = 1 - \frac{v_d}{mnh_1^d\log n},
	\end{align*}
	where $v_d$ is the volume of the unit ball in $d$ dimensions. Hence, for each individual $i\in [m]$ in the treatment group, we have
	\begin{align*}
	\bP(E_i) \triangleq \bP\left(\min_{j\in [n]} \|X_i^1 - X_j^0\|_2 \ge \frac{1}{h_1(mn\log n)^{1/d}} \right) \ge \left(1 - \frac{v_d}{mnh_1^d\log n}\right)^n \ge 1 - \frac{v_d}{mh_1^d\log n}. 
	\end{align*}
	
	Hence, the target quantity $\sum_{i=1}^m \1(E_i^c)$ is a Binomial random variable with number of observations $m$ and $\bP(E_i^c)\le v_d/(mh_1^d\log n)$. Using the first inequality of Lemma \ref{lemma.binomial} with $\delta = 1, \lambda = m\bP(E_i^c)$ gives
	\begin{align*}
	\bP\left(\sum_{i=1}^m \1(E_i^c) \ge \frac{2v_d}{h_1^d\log n} \right) \le \exp\left(- \frac{v_d}{3h_1^d\log n} \right). 
	\end{align*}
	Therefore, by using $h_1\le c(\log n)^{-2/d}$ with a small enough $c>0$, the claimed results hold for sufficiently large $n$. 
\end{proof}

By Lemma \ref{lemma:min_dist} we know that $|I| \le 1/(3h_1^d)$ with probability at least $1-n^{-3}$. We also define the set of \emph{good} indices as follows: 
\begin{align*}
I^\star = \left\{i\in [m]: \min_{j\in I} \|X_i^1 - X_j^1\|_2 > \frac{h_1}{2} \right\}. 
\end{align*}
We claim that for any $i\in I^\star$, we have $|\tau(X_i^1)|= \widetilde{\Omega}(h_1^{\beta_\tau})$. In fact, for any $j\notin I$, the definition of $I$ gives
\begin{align*}
\|X_i^1 - X_j^1\|_2^{\beta_\tau} + \min_{k \in [n]} \|X_j^1 - X_k^0\|_2^{\beta_\mu} \ge \min_{k \in [n]} \|X_j^1 - X_k^0\|_2^{\beta_\mu} > \left(\frac{1}{h_1(mn\log n)^{1/d}}\right)^{\beta_\mu} = \widetilde{\Omega}(h_1^{\beta_\tau}),
\end{align*}
where the last identity follows from the choice of $h_1$. As for the other case $j\in I$, the definition of $I^\star$ gives
\begin{align*}
\|X_i^1 - X_j^1\|_2^{\beta_\tau} + \min_{k \in [n]} \|X_j^1 - X_k^0\|_2^{\beta_\mu} \ge \|X_i^1 - X_j^1\|_2^{\beta_\tau} \ge \min_{j\in I} \|X_i^1 - X_j^1\|_2^{\beta_\tau} = \Omega(h_1^{\beta_\tau}),
\end{align*}
as claimed. Hence, all $|\tau(X_i^1)|$ with index $i\in I^\star$ have a large magnitude, and it remains to show that $|I^\star|$ is large with high probability. In fact, conditioning on $|I|\le 1/(3h_1^d)$, the following set
\begin{align*}
A := \left\{x\in [0,1]^d: \min_{j\in I} \|X_i^1 - x\|_2 \le \frac{h_1}{2} \right\}
\end{align*}
has $d$-dimensional volume at most $|I|\cdot h_1^d \le 1/3$. Note that $X_i^1$ is uniformly distributed, the region $A$ has probability at most $1/3$. Hence, by the standard Binomial tail bound (cf. Lemma \ref{lemma.binomial}), we conclude that with probability at least $1-2n^{-3}$, the number of treatment covariates $X_i^1$ falling into $A$ is at most $m/2$. Consequently, we have $|I^\star| \ge m/2$. In other words, with high probability, at least a constant fraction of $X_i^1$ satisfies $|\tau(X_i^1)|= \widetilde{\Omega}(h_1^{\beta_\tau})$, and therefore $\|\tau\|_1 = \widetilde{\Omega}(h_1^{\beta_\tau})$ with high probability as well. Now the two-point method applied to $(\mu_0, \tau)$ and $(\mu_0'\equiv 0, \tau'\equiv 0)$ together with the choice of $h_1$ gives the desired lower bound $\widetilde{\Omega}((\kappa/n^2)^{1/d(\beta_\mu^{-1} + \beta_\tau^{-1})})$. 

Next we prove the final lower bound $\widetilde{\Omega}((\kappa\sigma^2/n^2)^{{1}/{(2+d(\beta_\mu^{-1} + \beta_\tau^{-1}))}})$. To this end, we construct exponentially many hypotheses $(\mu_0^v, \tau^v)$ indexed by $v\in\{\pm 1\}^M$ and apply Fano's inequality. Specifically, let $h_1, h_2$ be the bandwidths given as follows:
\begin{align*}
h_1 &= \widetilde{\Theta}\left( \frac{1}{(mn)^{1/d}}\left(\frac{\sigma}{\sqrt{mn}}\right)^{-\frac{2\beta_\tau}{2\beta_\mu\beta_\tau+d(\beta_\mu+\beta_\tau)}}\right), \\
h_2 &=\widetilde{\Theta}\left( \left(\frac{\sigma}{\sqrt{mn}}\right)^{\frac{2\beta_\mu}{2\beta_\mu\beta_\tau+d(\beta_\mu+\beta_\tau)}}\right). 
\end{align*}
and $M = h_2^{-d}$. Fix any smooth function $g$ supported on $[0,1]^d$, then the standard dilation analysis yields that 
\begin{align*}
\tau^v(x) = c\cdot h_2^{\beta_{\tau}}\sum_{i=1}^{M} v_ig\left(\frac{x-x_i}{h_2}\right)
\end{align*}
for sufficiently small constant $c>0$ belongs to the H\"{o}lder ball $\calH_d(\beta_\tau)$ for all $v\in \{\pm 1\}^{M}$, where $x_1,\cdots,x_M$ are vertices of the small cubes so that $\{x_1,\cdots,x_M\} + [0,h_2]^d = [0,1]^d$. Note that
\begin{align*}
\|\tau^v - \tau^{v'} \|_1 = ch_2^{\beta_\tau}\|g\|_1\cdot d_{\text{H}}(v,v'),
\end{align*}
where $d_{\text{H}}(v,v') = M^{-1}\sum_{i=1}^M \1(v_i \neq v_i')$ is the normalized Hamming distance between binary vectors $v$ and $v'$. By the Gilbert-Varshamov bound, there exists $\calV_0 \subseteq \{\pm 1\}^M$ with $|\calV_0|\ge 2^{M/8}$ such that $d_{\text{H}}(v,v') \ge 1/5$ for all $v,v'\in \calV_0$. Let $V$ be a random variable uniformly distributed in $\calV_0$, and $Y$ be the collection of all outcomes in \eqref{eq.treatment_model} based on $\tau^v$ and $\mu_0^v$ defined below. The standard Fano-type arguments (see, e.g. \cite{tsybakov2008introduction}) give
\begin{align*}
R_{n,d,\beta_\mu,\beta_\tau,\sigma}^{\text{random}}(\kappa) = \Omega\left(h_2^{\beta_\tau}\left(1 - \frac{I(V;Y) + \log 2}{M/8} \right) \right). 
\end{align*}
Next we give an explicit construction of $\mu_0^v$ for each $v\in \{\pm 1\}^M$, and upper bound the mutual information $I(V;Y)$. 

The baseline functions $\mu_0^v$, also indexed by $v\in \{\pm 1\}^M$, are chosen to offset the differences of $\tau^v$ on the treatment covariates. Motivated by Lemma \ref{lemma.holder}, we choose the reference function
\begin{align*}
\tilde{\mu}_0(X_i^1) = c\cdot \min_{j\in [m]} \left(\|X_i^1 - X_j^1\|_2^{\beta_\mu} + \min_{k\in [n]} \|X_i^1 - X_k^0\|_2^{\beta_\mu} \right),
\end{align*}
and 
\begin{align*}
\mu_0^v(X_i^1) = -\text{sign}(\tau^v(X_i^1))\cdot \left(|\tau^v(X_{i}^1)| \wedge \tilde{\mu}_0(X_i^1) \right),
\end{align*}
and $\mu_0^v(X_i^0) = 0$ for any $i\in [n]$. As before, using similar arguments we conclude that the above value specifications satisfy the conditions of Lemma \ref{lemma.holder}, and thus each $\mu_0^v$ can be extended to a function in $\calH_d(\beta_\mu)$ for each $v$. Since distributions of the control outcomes do not depend on $v$, $\mu_0^v(X_i^1) + \tau^v(X_i^1) = 0$ whenever $\tilde{\mu}_0(X_i^1)\ge ch_2^{\beta_\tau}\|g\|_\infty\ge |\tau^v(X_i^1)|$, and all noises are normal distributed with variance $\sigma^2$, we have
\begin{align*}
I(V;Y) &\le \max_{v\neq v'\in \calV_0} D_{\text{KL}}(P_{Y|V=v}\| P_{Y|V=v'}) \\
&=  \max_{v\neq v'\in \calV_0} \sum_{i=1}^m \frac{(\mu_0^v(X_i^1) + \tau_0^v(X_i^1) - \mu_0^{v'}(X_i^1) - \tau_0^{v'}(X_i^1) )^2}{2\sigma^2} \\
&\le 4c^2h_2^{2\beta_\tau}\|g\|_\infty^2\cdot \sum_{i=1}^m \1(\tilde{\mu}_0(X_i^1) \le ch_2^{\beta_\tau}\|g\|_\infty ).
\end{align*}
By the choices of $h_1$ and $h_2$ above, we have $h_2^{\beta_\tau} = \widetilde{\Omega}((h_1(mn)^{1/d})^{-\beta_\mu})$. Hence, similar arguments of Lemma \ref{lemma:min_dist} lead to
\begin{align*}
\left| \left\{i\in [m]: \min_{j\in [n]}\|X_i^1 - X_j^0\|_2 \wedge \min_{j\neq i} \|X_i^1- X_j^1\|_2 \le  \frac{1}{h_1}\left(\frac{1}{mn\log n}\right)^{\frac{1}{d}} \right\} \right| = O(h_1^{-d})
\end{align*}
with probability at least $1-n^{-3}$. Also, for any $i\in [m]$ not belonging to the above set of indices, it is straightforward to see that $\tilde{\mu}_0(X_i^1) = \widetilde{\Omega}((h_1n^{2/d})^{-\beta_\mu})$. Hence, with probability at least $1-n^{-3}$ we have $\sum_{i=1}^n \1(\tilde{\mu}_0(X_i^1) \le ch_2^{\beta_\tau}\|g\|_\infty ) = O(h_1^{-d})$. Therefore, 
\begin{align*}
I(V;Y) = O\left(\frac{h_2^{2\beta_\tau}}{h_1^d\sigma^2}\right) = O\left(h_2^{-d} \right) = O(M)
\end{align*}
holds with high probability over the random covariates. Plugging the above upper bound of the mutual information into the Fano's inequality, we conclude that
\begin{align*}
R_{n,d,\beta_\mu,\beta_\tau,\sigma}^{\text{random}}(\kappa) = \widetilde{\Omega}(h_2^{\beta_\tau}) = \widetilde{\Omega}\left( \left(\frac{\kappa\sigma^2}{n^2}\right)^{\frac{1}{2+d(\beta_\mu^{-1} + \beta_\tau^{-1})}}\right),
\end{align*}
as desired.

\section{Proof of Main Lemmas}
\subsection{Proof of Lemma \ref{lemma:minimal_inequality}}
We first establish the following inequality: for any $\varepsilon>0$, there exists a constant $C_d>0$ depending only on $d$ such that
\begin{align}\label{eq:target}
\int_{[0,1]^d} f(x)\1(m[f](x) \le \varepsilon)dx \le C_d\varepsilon. 
\end{align}
To establish \eqref{eq:target}, we apply the maximal inequality in Lemma \ref{lemma.maximal} to the set $A = [-\sqrt{d},1+\sqrt{d}]^d$ and measures $\mu_1(dx) = f(x)\lambda(dx)$, $\mu_2(dx) = \lambda(dx)$, with $\lambda$ being the $d$-dimensional Lebesgue measure. Then clearly for every $x\in [0,1]^d$, we have
\begin{align*}
\sup_{\rho>0: B(x;\rho)\subseteq A} \left( \frac{\mu_2(B({x};\rho))}{\mu_1(B({x};\rho))} \right) \ge \frac{1}{m[f](x)}. 
\end{align*}
Consequently, choosing $t = 1/\varepsilon$ in Lemma \ref{lemma.maximal} gives
\begin{align*}
\int_{[0,1]^d} f(x)\1(m[f](x) \le \varepsilon)dx &\le \mu_1\left( \left\{x\in A:  \sup_{\rho>0: B(x;\rho)\subseteq A} \left( \frac{\mu_2(B({x};\rho))}{\mu_1(B({x};\rho))} \right) \ge \frac{1}{\varepsilon}  \right\} \right) \\
&\le C_d'\varepsilon\cdot \mu_2(A) = C_d\varepsilon, 
\end{align*}
establishing \eqref{eq:target}. 

Next we make use of \eqref{eq:target} to prove Lemma \ref{lemma:minimal_inequality}. For $\lambda>eC_d$, let $\varepsilon$ be a random variable following an exponential distribution with rate $\lambda$. Then clearly $\bP(m[f](x) \le \varepsilon) = \exp(-\lambda \cdot m[f](x))$, and taking expectation at both sides of \eqref{eq:target} leads to
\begin{align*}
\int_{[0,1]^d} f(x)\exp(-\lambda \cdot m[f](x))dx \le C_d\cdot \bE[\varepsilon] = \frac{C_d}{\lambda},
\end{align*}
as claimed. For $\lambda\le eC_d$, note that applying the above arguments gives that for any $t\ge 1$, 
\begin{align*}
\int_{[0,1]^d} f(x)\exp(-t\lambda \cdot m[f](x))dx \le  \frac{C_d}{t\lambda}. 
\end{align*}
Hence, using the fact that $f$ is a density and the H\"{o}lder's inequality, we have
\begin{align*}
\int_{[0,1]^d} f(x)\exp(-\lambda \cdot m[f](x))dx &\le \inf_{t\ge 1} \left(\int_{[0,1]^d} f(x)\exp(-t\lambda \cdot m[f](x))dx \right)^{\frac{1}{t}} \\
&\le \inf_{t\ge 1} \left(\frac{C_d}{t\lambda} \right)^{\frac{1}{t}} = \exp\left( - \frac{\lambda}{eC_d}\right),  
\end{align*}
where the minimizer in the last step is $t = eC_d/\lambda \ge 1$.

\subsection{Proof of Lemma \ref{lemma.holder}}
The necessity result simply follows from the H\"{o}lder ball condition, and it remains to prove the sufficiency. Without loss of generality we assume that $L=1$. We first show that given $\{(x_i,y_i)\}_{i\in [n]}$ satisfying the condition of Lemma \ref{lemma.holder}, for any other $x\in \bR^d \backslash \{x_i\}_{i\in [n]}$ there exists a value specification $y \in \RR$ of $f(x)$ such that the same condition holds for the $(n+1)$ points $\{(x_i,y_i)\}_{i\in [n]} \cup \{(x,y)\}$, i.e.
\begin{align*}
|y_i - y|
\le \|x_i - x\|_2^{\beta}, \quad \forall i \in [n]. 
\end{align*}
To prove the existence of $y$, it suffices to show that 
\begin{align*}
\min_{i\in [n]}\left\{y_i +  \|x_i - x\|_2^{\beta}\right\}
\ge \max_{j\in [n]}\left\{y_j -  \|x_j - x\|_2^{\beta}\right\}.
\end{align*}
In fact, for $i$, $j \in [n]$,
\begin{align*}
\left(y_i +  \|x_i - x\|_2^{\beta}\right)
- \left(y_j -  \|x_j - x\|_2^{\beta}\right)
&= (y_i - y_j) + \|x_i - x\|_2^{\beta} +  \|x_j - x\|_2^{\beta} \\
&\ge - \|x_i - x_j\|_2^{\beta} + \|x_i - x\|_2^{\beta} +  \|x_j - x\|_2^{\beta} \\
&\ge 0.
\end{align*}
Hence, we may construct the $(n+1)$-th point based on $n$ points, and repeating this procedure leads to a construction of $f$ on any countable dense subset $\calB \subseteq \bR^d$. Now the proof is completed by taking
\begin{align*}
f(x) = \lim_{x_n \in \calB, x_n\to x} f(x_n). 
\end{align*}

\subsection{Proof of Lemma \ref{lemma.weights}}
For $t = 0$, the lemma is obvious. In the following we consider $t \ge 1$ and write the linear equations in the matrix form
\begin{align*}
\label{eq:linearEquationMatrix}
\begin{bmatrix}
1 &1  &\cdots &1 \\
z_0 &z_1 &\cdots &z_t \\
\vdots & \vdots &\vdots &\vdots \\
z_0^t &z_1^t &\cdots &z_t^t \\
\end{bmatrix}
\begin{bmatrix}
w_{0} \\
w_{1} \\
\vdots \\
w_{t}
\end{bmatrix}
= 
\begin{bmatrix}
1 \\
0 \\
\vdots \\
0
\end{bmatrix}.
\end{align*}
where $z_j= x - x_j$, $0 \le j \le t$. Since the above square matrix is a Vandermonde matrix, the weight vector $(w_0,\cdots,w_t)$ exists and is unique. Now by the Cramer's rule, for $i\ge 1$ we have
\begin{align*}
|w_i| &= \left| \frac{\det \begin{bmatrix}
	1 & \cdots & 1 & \cdots & 1 \\
	z_0 & \cdots & 0 & \cdots & z_t \\
	\vdots & \vdots & \vdots & \vdots & \vdots \\
	z_0^t & \cdots & 0 & \cdots & z_t^t
	\end{bmatrix}
}{\det \begin{bmatrix}
	1 & \cdots & 1 & \cdots & 1 \\
	z_0 & \cdots & z_i & \cdots & z_t \\
	\vdots & \vdots & \vdots & \vdots & \vdots \\
	z_0^t & \cdots & z_i^t & \cdots & z_t^t
	\end{bmatrix}} \right| = \prod_{0\le j\le t, j\neq i} \left| \frac{z_j}{z_j - z_i} \right| \\
&= \frac{\Delta}{|z_0 - z_i|}\cdot \prod_{1\le j\le t, j\neq i} \left|\frac{z_j}{z_j - z_i} \right| \le m\Delta\cdot (t+1)^{t-1},
\end{align*}
where the last inequality follows from $|z_j|\le (t+1)/m$ and $|z_j - z_i|\ge 1/m$ for $i\neq j$. The proof is completed by noting that
\begin{align*}
|w_0| \le 1 + \sum_{i=1}^t |w_i| \le 1 + m\Delta\cdot t(t+1)^{t-1} \le 1 + t(t+1)^{t-1} =: C. 
\end{align*}

\bibliography{ref}
\bibliographystyle{alpha}

\end{document}